\documentclass[final,hidelinks,onefignum,onetabnum]{siamart250211}

\usepackage{amsmath,amsfonts,amssymb}
\usepackage{algorithm}
\usepackage{algpseudocode}
\usepackage{pgfplots}
\usepackage{pgfplotstable}
\usepackage{booktabs}
\newcommand{\qedhere}{}

\usetikzlibrary{calc}

\usepackage{lipsum}
\usepackage{amsfonts}
\usepackage{graphicx}
\usepackage{epstopdf}
\ifpdf
  \DeclareGraphicsExtensions{.eps,.pdf,.png,.jpg}
\else
  \DeclareGraphicsExtensions{.eps}
\fi

\newsiamremark{remark}{Remark}
\newsiamremark{hypothesis}{Hypothesis}
\crefname{hypothesis}{Hypothesis}{Hypotheses}
\newsiamthm{claim}{Claim}
\newsiamremark{fact}{Fact}
\crefname{fact}{Fact}{Facts}

\headers{Error for block Krylov approximation of matrix functions}{S. Massei, L. Robol}

\title{Error formulas for block rational Krylov approximations of matrix functions\thanks{Both authors are members of INdAM/GNCS. \funding{
  The authors have been partially supported by the 
  MIUR Excellence Department Project awarded to the Department of Mathematics, University of Pisa, CUP I57G22000700001.
  The work of Leonardo Robol was partially supported by the National Research Center in High Performance Computing, Big Data and Quantum Computing (CN1 -- Spoke 6), by the MIUR Excellence Department Project awarded to the Department of Mathematics, University of Pisa (CUP I57G22000700001), 
and by the Italian Ministry of University and Research (MUR) through the PRIN 2022 ``MOLE: Manifold constrained Optimization and LEarning'',  code: 2022ZK5ME7 MUR D.D. financing decree n. 20428 of November 6th, 2024 (CUP I53C24002260006). 
The work of Stefano Massei has been supported
by the PRIN 2022 Project ``Low-rank Structures and Numerical Methods in Matrix and Tensor Computations and their Application''. 
}}}

\author{Stefano Massei\thanks{Department of Mathematics, University of Pisa, L.go B. Pontecorvo, 5, 56127, Pisa, Italy. 
  (\email{stefano.massei@unipi.it}, \email{leonardo.robol@unipi.it}).}
\and Leonardo Robol\footnotemark[2]}

\usepackage{amsopn}

\newcommand{\res}{\mathrm{Res}}
\newcommand{\U}{\mathbf U}

\newcommand{\Z}{\mathbf Z}

\newcommand{\V}{\mathbf V}

\newcommand{\hbradu}{{\underline H}}
\newcommand{\kbradu}{{\underline K}}
\newcommand{\hbrad}{{H}}
\newcommand{\kbrad}{{K}}

\newcommand{\bspan}{\mathrm{blockspan}}
\newcommand{\norm}[1]{\lVert#1\rVert}

\newcommand{\upd}[1]{{#1}} 

\ifpdf
\hypersetup{
  pdftitle={Error formulas for block rational Krylov approximations of matrix functions},
  pdfauthor={S. Massei, L. Robol}
}
\fi

\begin{document}

\maketitle

\begin{abstract}
  This paper investigates explicit expressions for the error associated with the block rational Krylov approximation of matrix functions. Two formulas are proposed, both derived from  characterizations of the block FOM residual. The first formula employs a block generalization of the residual polynomial, while the second leverages the block collinearity of the residuals. A posteriori error bounds based on the knowledge 
  of spectral information of the argument are derived and tested on a set of examples. Notably, both error formulas and their corresponding upper bounds do not require  the use of quadratures for their practical evaluation.
\end{abstract}


	
\section{Introduction}
This work focuses on  evaluating the action of a matrix function on a block vector, i.e., we look at the numerical approximation of the quantity 
\begin{equation}\label{eq:mat-fun}		
 f(A)B, 
\end{equation}
where $A\in\mathbb C^{n\times n}$, $f:\mathbb C\rightarrow\mathbb C$ is analytic around the eigenvalues of $A$, and $B\in\mathbb C^{n\times s}$. Computing $f(A)B$ is an advanced linear algebra task that is crucial for time integration of ODEs~\cite{hochbruck2010exponential}, trace estimation~\cite{ubaru2017fast}, network analysis~\cite{benzi2020matrix}, tensor equations~\cite{kressner2010krylov}, and many more applications (see \cite[Section~2]{higham-matfun} and the references therein). 

Krylov subspace projection methods are the workhorse algorithms for evaluating matrix function related expressions like those in  \eqref{eq:mat-fun}. The latter  methods iteratively build a
sequence of nested low-dimensional Krylov subspaces and extract an approximation by imposing either a Galerkin or 
a Petrov-Galerkin condition on the error. For large matrices, 
this approach is much more efficient than approximating $f(A)$ 
and then multiplying it with $B$. 
The smaller the dimension of the Krylov subspace, the larger 
the computational gain with respect to general purpose algorithms based on dense linear algebra. 
In particular, it is important to monitor the error of the approximation as the subspace is expanded, to not overestimate the dimension of the Krylov subspace needed to reach a target accuracy. 
However, there is not a natural way to check the convergence 
apart from explicitly forming the error. This is in contrast with 
Krylov methods for linear systems, where the residual can be 
cheaply evaluated. Only for specific matrix functions, quantities
playing similar roles to the residual \upd{have} been considered in 
the literature, see e.g., \cite{botchev2013residual,botchev2021residual, botchev24}.

Many studies propose a-priori and a-posteriori error bounds for Krylov-based procedures, with the ultimate goal of providing suitable stopping criteria \cite{beckermann2009error,chen2022error,simunec2023error}, pole selection strategies \cite{guttelblack,massei2021rational,moret2019krylov}, and 
restarting techniques \cite{frommer2014convergence,frommer2016error}. While the latter works  target the single vector scenario ($s=1$), here we specifically address the  block case $s>1$. This is a less explored setting, partly because one can extend in a straightforward manner the results of the single vector case by treating each column in $B$  independently. For instance, an immediate consequence of the polynomial exactness property \cite[Chapter~13.2]{higham-matfun} is that the Frobenius norm of the error associated with the Galerkin approximation of $f(A)B$ onto $\mathrm{colspan}(B, AB,\dots ,A^\ell B)$ can be related to the best polynomial approximation error 
on a spectral set for $A$ \cite{crouzeix2017numerical,trefethen2020spectra}.
When $s>1$, such bounds are not able to capture the true convergence rate of the algorithm, in general, as they only depend on the norm of $B$, without accounting for other features of the block vector, \upd{see also the numerical test reported in Figure~\ref{fig:test-intro}}. For this reason, there is a body of literature dealing 
specifically with the block case. 

\begin{figure}
	\centering
	\begin{tikzpicture}
		\begin{semilogyaxis}[
			width=0.7\textwidth,
			height=0.45\textwidth,
			xlabel={$j$},
			ylabel={},
			grid=major,
			legend style={ at={(1.02,1)}, anchor=north west},
			mark size=2pt
			]
			
			
			\addplot[
			no marks, orange, very thick
			] table[x index=0, y index=3] {test_intro.dat};
			\addlegendentry{A-priori bound}

			\addplot[
			no marks, blue, very thick
			] table[x index=0, y index=2] {test_intro.dat};
			\addlegendentry{Corollary~\ref{cor:residual-matfun2}}
			
			\addplot[
			no marks, red, very thick
			] table[x index=0, y index=4] {test_intro.dat};
			\addlegendentry{Corollary~\ref{cor:aposteriori-matfun1}}
			
			\addplot[
			mark=square*, teal, very thick
			] table[x index=0, y index=1] {test_intro.dat};
			\addlegendentry{$\| \exp(A) B  - F_j \|_F$}
			
		\end{semilogyaxis}
	\end{tikzpicture}
	\caption{Approximation error and upper bounds for the computation of the 
		quantity $\exp(A) B$ via the block polynomial Krylov subspace method, with $A \in \mathbb C^{400 \times 400}$
		diagonal with logarithmically spaced eigenvalues over $\Omega=[-16, -10^{-4}]$, 
		and $B \in \mathbb C^{400 \times 2}$ with entries 
		$B_{i1} = 0.95^{i}, B_{i2} = 0.99^i$, scaled so that $\norm{B}_F=1$. The teal curve with square marks indicates the Frobenius norm of the error of the Krylov approximation. The orange curve represents the 
		a-priori bound $\min_{\mathrm{deg}(p)< j}\max_{z\in\Omega}|p(z)-\exp(z)|$, computed via the \texttt{minimax} function of the Chebfun toolbox~\cite{chebfun}. The blue and the red curves, are a posteriori upper bounds based on our Corollary~\ref{cor:aposteriori-matfun1}, and Corollary~\ref{cor:residual-matfun2}, respectively.}\label{fig:test-intro}
\end{figure}
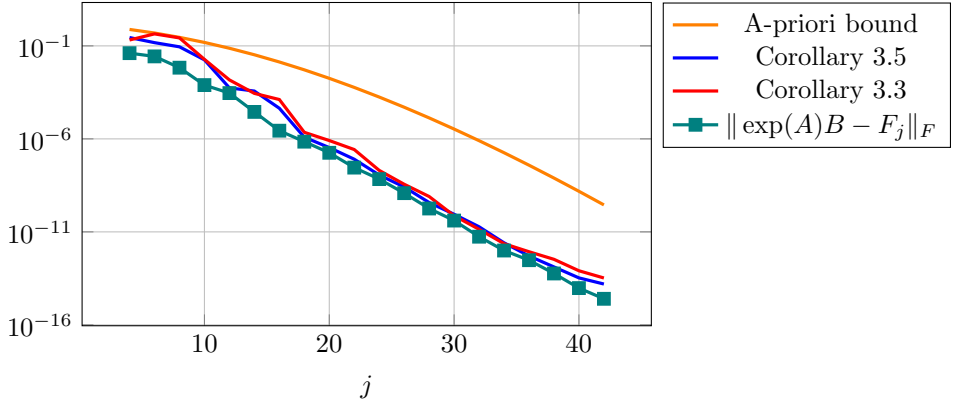

\begin{paragraph}{Literature review}
The works that tackle the block scenario analyze the approximation properties of the block Krylov subspaces as subspaces (or $\mathbb C^{b \times b}$-submodules) of $\mathbb C^{n\times b}$
\cite{frommer2017block,gutknecht2007block}.  Note that when $f(z)=z^{-1}$ we are concerned with Krylov methods for solving linear systems with multiple right-hand sides~\cite{soodhalter2015block}. By 
considering the contour integral defining $f(A) B$, one can draw 
a connection between matrix functions and parameter dependent 
shifted linear systems. This observation has been exploited to 
use error bounds for block linear systems 
to design error estimators and restarting techniques for 
functions of matrices \cite{frommer2017block,frommer2020block,xu2024posteriori}. These works provide integral representations 
of the error that involve the residual of linear systems 
with matrix $zI - A$; by means of a quadrature, this leads 
to a formula for the approximation error that requires checking 
the residuals at a few points $z \in \mathbb C$. 
Shifted block linear systems are also relevant for the 
approximation of rational matrix-valued functions, such as the transfer function of multi-input multi-output 
linear time invariant systems \cite{antoulas2001approximation}. Within the context of 
interpolatory model order reduction and moment-matching, error 
bounds can be used to design greedy selection strategies for 
the interpolation points \cite{druskin2014adaptive}. 
Other approaches based on Gauss quadratures are 
used to provide error bounds for quantities like
$B^T f(A) B$, with $A$ symmetric positive definite~\cite{fenunetworks,golub-meurant,zimmerling2025monotonicity}.

\end{paragraph}
\begin{paragraph}{Contribution}
This paper makes several contributions to the block Krylov subspace setting. 
First, we propose two novel formulas for
the approximation error $E_j = f(A)B - F_j$, where $F_j$ 
is the Galerkin (or Petrov-Galerkin) approximation of $f(A)B$
obtained by means of $j$ steps of a block rational Krylov method.
These formulas are based on different expressions of the block FOM residual, which are integrated on a contour enclosing the spectrum of $A$ using the residue theorem.
The first formula ---contained in Corollary~\ref{cor:f-res}--- relies on a block generalization of the characteristic polynomial of the projection of $A$ onto the block Krylov subspace. 
By means of Keldysh's theorem, the formula allows us to relate the 
spectrum of $A$ with spectral information of the block characteristic polynomial.
The second formula ---from Corollary~\ref{cor:residual-matfun2}--- exploits the collinearity of the block FOM residuals
to have a compact expression that only involves the block upper 
Hessenberg matrices appearing in the rational block Arnoldi 
decomposition. These two formulas are used to 
derive a posteriori upper bounds 
for $\| E_j \|_F$, and present numerical tests to assess their effectiveness.
Other contributions concern the study of the properties 
of block characteristic polynomials \cite{lund2018new}, for which 
we discuss uniqueness, behavior under similarity transformations 
or change of initial vectors, and practical ways of evaluating their 
action. \upd{We show that the Hessenberg pencils associated with
the rational Arnoldi decomposition are linearizations of 
the block characteristic polynomials of the projected 
matrix.}
\end{paragraph}

\paragraph{Outline}	In Section~\ref{sec:shifted-linear-systems}, we introduce the notation 
and assumptions for block Krylov subspaces and the Petrov-Galerkin 
approximation for shifted linear systems. We introduce and 
study the main properties of block characteristic polynomials, 
and use them to characterize the matrix polynomials expressing the 
error and the residual associated with 
a block FOM approximation of shifted linear systems. Then, we 
study the collinearity of the residuals, and we retrieve a 
second formula based on this property. The proposed formulas 
address both polynomial and rational block Krylov subspaces.
Section~\ref{sec:matfun} is devoted to  error 
formulas and related a posteriori bounds for the approximation
of $f(A) B$. \upd{The bounds are validated through numerical examples.} Section~\ref{sec:computational} \upd{contains proof of key properties concerning block characteristic polynomial related with block upper Hessenberg matrices and pencils.}  Section~\ref{sec:conclusions} includes some conclusions and possible research directions.
\section{Petrov-Galerkin approximation of shifted linear systems with block Krylov subspaces}
\label{sec:shifted-linear-systems}
Let us suppose to have a matrix $A \in \mathbb C^{n \times n}$, and a block vector 
$B \in \mathbb C^{n \times s}$, with $s>1$; our aim is approximating the map $X_B(z) = (zI-A)^{-1}B$ for $z \in \mathbb C$.
In the rest of the paper, to highlight  the single vector case, whenever we consider  $s=1$ we replace the capital letter $B$ with the lower case letter $b$.

To build an approximation to $X_B(z)$, we consider $j$ iterations of the block Arnoldi method with $A$, and starting block vector $B$;  we assume that  the   procedure does not breakdown nor encounter deflation\footnote{\upd{By deflation we mean that a block vector generated by the Arnoldi procedure is not full rank.}}. This yields a matrix with orthonormal columns  
\[
\U_j := \begin{bmatrix}
  \\ 
  U_1 & \dots & U_j \\ 
  \\
\end{bmatrix} \in \mathbb C^{n \times js}, 
\]
with $B=U_1 R_B$, for some $R_B\in\mathbb C^{s\times s}$, and such that its block columns form a basis  for the 
block Krylov subspace \begin{align*}\mathcal K^\square_j(A, B)&:=\bspan\{B,AB,\dots,A^{j-1}B\}\subset \mathbb C^{n\times s},
\end{align*} 
where $\bspan$ indicates the set of linear combinations of the block vectors with $s\times s$ coefficients:
$$
\bspan(U_1,\dots, U_j):=\left\{ \sum_{i=1}^j U_iP_i,\quad P_i\in\mathbb C^{s\times s}\right\}.
$$
We refer to the matrix $\mathbf U_j$ (or to the set of its block columns $U_1,\dots,U_j$) as a \emph{block basis}. More specifically, 
$\mathcal K^\square_j(A, B)$ can be described as a free module 
over the ring $\mathbb C^{s \times s}$, of rank or dimension $j$.
Alternatively, $\mathcal K^\square_j(A, B)$ can be seen as a vector 
subspace of $\mathbb C^{n \times s}$, of dimension $s^2j$.
In fact, any block vector in $\mathcal K^\square_j(A, B)$ has $s$ columns, 
and each of those belong to the $js$-dimensional subspace
$\mathrm{colspan}(U_1, \ldots, U_j)$. 
Moreover, we say that  the block basis $\mathbf U_j$
is unitary/orthogonal when $\mathbf U_j^* \mathbf U_j=I_{js}$. 

The unitary block basis $\U_j$ satisfies the block Arnoldi relation
\begin{align*}
  A \U_{j} &= \mathbf U_{j+1} \underline H_{U,j} = 
  \U_j H_{U,j} + U_{j+1} \Gamma_{j+1}^U E_j^*, 
\end{align*}
where $H_{U,j}\in \mathbb C^{js \times js}$ is block upper Hessenberg with $s\times s$ blocks
  $$
H_{U,j}:=
\begin{bmatrix}
  \Phi_{1}^{U} & \Xi_{1,2}^{U}&\dots&\Xi_{1,j}^{U} \\ 
  \Gamma_2^{U} & \Phi_2^{U} & \ddots &\vdots \\ 
  & \ddots & \ddots & \Xi_{j-1,j}^{U} \\ 
  & & \Gamma_{j}^{U} & \Phi_j^{U} \\
\end{bmatrix},
$$
$E_j = e_j \otimes I_s$, $e_j$ is the $j$th vector of the canonical basis,  and 
$\underline H_{U,j}$ is the augmented matrix 
\[
\underline H_{U,j} := \begin{bmatrix}
  H_{U,j} \\ 
  \Gamma_{j+1}^U E_j^*
\end{bmatrix} \in \mathbb C^{(j+1)s \times js}. 
\]
Our goal is to find an approximation $X_{B,j}(z)\approx X_B(z)$ that, for all $z\in\mathbb C$, belongs to $\mathcal K^\square_j(A,B)$, and satisfies the Petrov-Galerkin condition
  $\Z_j^*(B-(zI-A)X_{B,j}(z)) = 0_{s \times s}$ for a chosen 
  unitary basis $\Z_j \in \mathbb C^{n \times js}$.    
To characterize the Petrov-Galerkin approximation, we make 
the following assumptions that will be required throughout the paper.
\medskip

\noindent \fbox{
	\parbox{.95\linewidth}{
		\begin{center}
			\textbf{Assumptions 1}
		\end{center}
		\begin{enumerate}
			\item The block Arnoldi procedure has been run for $j+1$ steps for $A$, and 
			starting block vector $B$, without encountering breakdown nor deflations. 
			\item The block unitary bases $\Z_j, \U_j \in \mathbb{C}^{n \times js}$  
			are such that $\Z_j^* \U_j$ is invertible. 
\end{enumerate}}}
\medskip 

Since $\Z_j^* \Z_j = I_{js}$ and $\Z_j^* \U_j$ is invertible, then
the Petrov-Galerkin condition implies that $X_{B,j}(z)=\U_jY_{B,j}(z)$ for some $Y_{B,j}(z)\in\mathbb C^{js\times s}$, and \begin{align*}
0&=\Z_j^*(B-(zI-A)\U_jY_{B,j}(z))\\
&\Leftrightarrow (z\Z_j^*\U_j-\Z_j^*A\U_j)Y_{B,j}(z)=\Z_j^*B\\
&\Leftrightarrow Y_{B,j}(z)= (zI-(\Z_j^*\U_j)^{-1}\Z_j^*A\U_j)^{-1} (\Z_j^*\U_j)^{-1}\Z_j^* B.
\end{align*}
Observe that $(\Z_j^*\U_j)^{-1}\Z_j^* B=(\Z_j^*\U_j)^{-1}\Z_j^* \U_jE_1R_B= E_1R_B$, and this leads to
\begin{equation}\label{eq:XBjarnoldi}
  \begin{split}
    X_{B,j}(z)&:=\U_j (zI-(\Z_j^*\U_j)^{-1}\Z_j^*A\U_j)^{-1}E_1 R_B.
  \end{split}
\end{equation}
  For convenience, we introduce the Petrov-Galerkin 
  and Galerkin projections of $A$ as:
  \[
    A^{U,Z}_j := (\Z_j \U_j)^{-1} \Z_j^* A \U_j, \quad
    A^{U}_j := A^{U,U}_j = \U_j^* A \U_j.
  \]
  Note that, in the Galerkin case, we have 
  $A^{U}_j = H_{U,j}$; this will not be the case 
  when dealing with rational Krylov subspaces.
  Moreover, the Arnoldi relation implies
\begin{equation}\label{eq:upper-hess}
  \begin{split}
    A_j^{U,Z} = (\Z_j^*\U_j)^{-1}\Z_j^*A\U_j &= H_{U,j} + (\Z_j^*\U_j)^{-1}\Z_j^*U_{j+1}\Gamma_{j+1}^UE_j^*.
  \end{split}
\end{equation}
Thus, the matrix $A_j^{U,Z}$, coincides with the block upper Hessenberg matrix $H_{U,j}$, apart from the last block column. In particular, it is block upper Hessenberg as well. 

%
We remark that Assumptions 1 imply that
the columns of $\U_{j+1}$ span a $(j+1)s$-dimensional subspace 
of $\mathbb C^{n}$,  
and the subdiagonal blocks $\Gamma_i^U$ 
are invertible for $i = 2, \ldots, j+1$. 
To show the second claim, assume by contradiction that 
$1 \leq i \leq j$ is the first index 
such that $\Gamma_i$ is not invertible. Then, we have 
\begin{align*}
	\mathcal K^\square_{i+1}(A, B) &= \bspan(
	U_1, \ldots, U_i, AU_i
	) \\
	&= \bspan\left(
	U_1, \ldots, U_i,  
	U_i \Phi_i + U_{i+1} \Gamma_i + \sum_{k < i} U_k \Xi_{k,i}
	\right) \\ 
	&= \bspan(
	U_1, \ldots, U_i, U_{i+1} \Gamma_i
	).
\end{align*}
Since $\Gamma_i$ is not full rank, the 
columns of $\left[ U_1, \ldots, U_i, U_{i+1} \Gamma_i \right]$
span a subspace of dimension strictly 
smaller than $(i+1)s$, and therefore there cannot be 
$(j + 1)s$ linear independent columns among those of 
$\U_{j+1}$.
\footnote{
	See also \cite{gutknecht2009block,simoncini1996hybrid} for early works 
	on block Krylov methods.
} 

Finally, we introduce the residual map associated with \eqref{eq:XBjarnoldi}: 
\begin{align}\label{eq:resj}
  \res_{B,j}(z):=B - (zI - A)X_{B,j}(z).
\end{align}
In the case $s=1$, closed formulas for the errors and residuals associated with $X_{b,j}(z)$, are known, see~\cite[Section 4.5]{gutknecht},\cite[Section 7.4.1]{saad}, \cite[Lemma 2.4]{lin2021transfer}, and are related with the characteristic polynomial $\Lambda(\lambda)$ of $A_j^{U,Z}$. More precisely, we have that
\begin{equation} \label{eq:residual-s=1}
  \begin{split}
    \res_{b,j}(z) &:=b-(zI-A)X_{b,j}(z)= \frac{\Lambda(A)}{\Lambda(z)}  b.
  \end{split}
\end{equation}
We emphasize that choosing $\Z_j=\U_j$ boils down to the Galerkin case, where $X_{B,j}(z)$ is the solution returned by the BFOM algorithm~\cite{frommer2017block,frommer2020block}. The goal of the following section is to provide a non-trivial generalization of formula \eqref{eq:residual-s=1} to the case $s> 1$ that ---in the Galerkin case--- is a novel explicit formula for the residual of BFOM. 
 To get there, we will recall a suitable formalism to operate with block vectors, and prove a few auxiliary \upd{lemmas}.


\subsection{Matrix polynomials and their action on block vectors}  
In the case $s=1$, a Krylov subspace $\mathcal K_j(A, b)$ starting from a single vector $b$, can be characterized as the set $\mathcal K_j(A, b)=\{p(A)b:\ \mathrm{deg}\ p<j\}$. When $s>1$, the block Krylov subspace $\mathcal K^\square_j(A,B)$ can be
written  in terms of  actions of the powers of $A$:
$$
\mathcal K^\square_j(A,B)=\left\{\sum_{i=0}^{j-1}  A^i B P_i:\ P_i\in\mathbb C^{s\times s}\right\}.
$$
The above set is described more similarly to the single vector case by considering the set of matrix polynomials $\mathbb C[\lambda]^{s\times s}:=\{Q(\lambda)=\sum_{i=0}^{k}\lambda^i Q_i:\ Q_i\in\mathbb C^{s\times s},\ k\in\mathbb N\}$, and introducing the operator\footnote{\upd{The operator $\circ$ has been introduced in \cite{kent1989chebyshev}, and the author attributes this notation to Gregg.}} $\circ$ defined as
\begin{equation}\label{eq:circ}
  Q(A)\circ B := \sum_{i=0}^{j-1}  A^i B Q_i,
\end{equation}
so that
$
  \mathcal K^\square_j(A,B)=\left\{P(A)\circ B:\ P(\lambda)\in\mathbb C[\lambda]^{s\times s},\  \mathrm{deg}\ P(\lambda)<j \right\}.
$
As for usual Krylov subspaces, the action of a polynomial of $A$ on $B$ is strictly linked to the action of the polynomial of the projected matrix on the projected block vector. More precisely, 
Theorem~2.7 from \cite{frommer2020block} rewrites
$P(A) \circ B$, with $\deg(P)<j$, as the evaluation of $P(\lambda)$ at 
any $sj \times sj$ matrix that coincides with $H_{U,j}$, apart from 
the last block column. The latter property implies 
$
  P(A) \circ B = \U_j P(A_j^{U,Z}) \circ (E_1 R_B)
$
for all matrix polynomials of degree strictly less than $j$. 
For our purposes, we need to slightly extend this result 
to cover the case when $P(\lambda)$ has degree $j$. We prove the statement also for the case  
of degree lower than $j$, because the proof is shorter than the one in \cite{frommer2020block}, and 
may be of its own interest. Before stating the result we define the following oblique  projector onto  $\mathrm{colspan}(\Z_j)^\perp$
$$
\Pi_{U,Z}:= I - \U_j (\Z_j^* \U_j)^{-1} \Z_j^*,
$$
and note that, in the Galerkin case, it is an orthogonal projector onto $\mathrm{colspan}(\U_j)^\perp$.
\begin{lemma} \label{lem:poly-lanczos}
  Let $P(\lambda) = \lambda^kP_k  + \ldots + \lambda P_1+P_0$ 
  be an $s \times s$ matrix polynomial of degree $k$. Under 
  Assumptions 1, 
  \begin{align*}
    P(A) \circ B &= \begin{cases}
      \mathbf U_j \Big[
      P(A_j^{U,Z}) \circ (E_1 R_B)
      \Big]  & j > k \\[.2cm] 
      \mathbf U_j \Big[ 
      P(A_j^{U,Z}) \circ (E_1R_B)
      \Big] + \Upsilon_j^U R_B P_j & j = k, \\ 
    \end{cases}
  \end{align*}
  where 
  $\Upsilon_j^U = \Pi_{U,Z}U_{j+1} \Gamma_{j+1}^U \ldots \Gamma_{2}^U$ has full rank. 
\end{lemma}

\begin{proof}
  We can restrict to the case 
  $P(\lambda) = P^{[k]}(\lambda) := \lambda^k I_s$, 
  since by linearity
  \[
  P(A) \circ B = \sum_{i = 0}^k A^i B P_i = 
  \sum_{i = 0}^k (P^{[i]}(A) \circ B) P_i, 
  \]
  We proceed by induction. For 
  $k = 0$, we have $P(A)\circ B= B$, and the claim follows by noting that 
  $B = U_1 R_B = \mathbf U_j E_1 R_B$. 
  For $0 < k \leq j$, we have 
  \[
  P^{[k]}(A) \circ B = A [ P^{[k-1]}(A) \circ B ] = 
  A \mathbf U_j \left[ P^{[k-1]}(A_j^{U,Z}) \circ (E_1 R_B)\right]. 
  \]
  Using the Arnoldi relation, we get 
  \begin{align*}
    A \mathbf U_j \left[ P^{[k-1]}(A_j^{U,Z}) \circ (E_1 R_B) \right] &= 
    ( \mathbf U_j H_{U,j} + U_{j+1} \Gamma_{j+1}^U E_j^* ) (A_j^{U,Z})^{k-1} E_1 R_B.
  \end{align*}
  We recall from \eqref{eq:upper-hess} that $A_j^{U,Z} = H_{U,j} + (\Z_j^* \U_j)^{-1} \mathbf Z_j^* U_{j+1} \Gamma_{j+1}^U E_j^*$, 
  so that 
  \begin{align*}
    P^{[k]}(A) \circ B &=   
    \mathbf U_j (A_j^{U,Z})^k E_1 R_B  + \Pi_{U,Z}
    U_{j+1} \Gamma_{j+1}^U E_j^* (A_j^{U,Z})^{k-1} E_1 R_B \\ 
    &= \mathbf U_j P^{[k]}(A_j^{U,Z}) E_1 R_B  + \Pi_{U,Z}
    U_{j+1} \Gamma_{j+1}^U E_j^* (A_j^{U,Z})^{k-1} E_1 R_B
  \end{align*}
  To get the claim, note that if $k < j$ then
  $E_j^* (A_j^{U,Z})^{k-1} E_1 = 0$, and if $k = j$, we have 
  $E_j^* (A_j^{U,Z})^{k-1} E_1 = \Gamma_j^U \ldots \Gamma_2^U$. 
  
  As the Arnoldi method does not encounter breakdown nor deflation at step $j+1$, then $\Upsilon_j^U$ is full rank if and only if $(I - \U_j (\Z_j^* \U_j)^{-1} \Z_j^*)
  U_{j+1}$ is full rank. The latter is a linear combination of $U_{j+1}$ and 
  a block vector in $\mathcal K^\square_j(A, B)$, and is full rank since $U_{j+1}$ is orthogonal to $\mathcal K^\square_j(A, B)$. 
\end{proof}
Since our goal is to describe the residual of a Petrov-Galerkin approximation, we characterize all the block vectors that belong to a block Krylov subspace and are block orthogonal to \upd{$\mathcal Z:=\bspan(Z_1,\dots,Z_j)$, with $\mathbf{Z}_j=\left[Z_1\ \dots\ Z_j\right]$}.  To state the next result, we need 
to recall a generalization of 
characteristic polynomial for a block 
matrix with $s \times s$ blocks. 

\begin{definition}[Definition 2.23 in \cite{lund2018new}]
  Let $N$ be a $js \times js$ matrix with
  $s \times s$ blocks, and $W \in \mathbb C^{js \times s}$ be a block vector. Then, 
  an $s \times s$ matrix polynomial $P(\lambda)$ of degree $j$ is a 
  \emph{block characteristic polynomial
    of $N$ with respect to $W$} if 
  $P(N) \circ W = 0$. 
\end{definition}
\begin{remark}\label{rem:less-magic}
  $P(\lambda)$ is a block characteristic polynomial for $N$ with respect to $WX$ for an invertible $X\in\mathbb C^{s\times s}$, if and only if $\widetilde P(\lambda):= XP(\lambda)X^{-1}$ is a block characteristic polynomial for $N$ with respect to $W$. $P(\lambda)$ is monic if and only if $\widetilde P(\lambda)$ is monic.
\end{remark}

\begin{lemma} \label{lem:ortho-poly}
  Let $R^U \in \mathcal K^\square_{j+1}(A, B) \cap \mathcal Z^\perp$ be a nonzero block vector. Under Assumptions 1, 
  there \upd{exists} a block characteristic polynomial $P^U(\lambda)$ 
  for $A_j^{U,Z}$ with respect to $E_1R_B$ such that 
  $R^U = P^U(A) \circ B$.
\end{lemma}

\begin{proof}
  Since $R^U \in \mathcal K^\square_{j+1}(A, B)$, 
  we can write it as $R^U = P^U(A) \circ B$ for a matrix polynomial $P^U(\lambda)$
  of degree at most $j$. In view of Lemma~\ref{lem:poly-lanczos}, 
  we have 
  \[
  P^U(A) \circ B = \U_j \Big[ 
  P^U(A_j^{U,Z}) \circ (E_1R_B)
  \Big] + \Upsilon_j^U R_B P_j^U , 
  \]
  where we have combined the two cases by allowing $P_j^U = 0$. 
  Imposing the orthogonality relation by left multiplying 
  with $\Z_j^*$, we get 
  \[
  0 = \Z_j^*\left[ 
  P^U(A) \circ B \right] = 
  \Z_j^*\U_jP^U(A_j^{U,Z}) \circ (E_1R_B), 
  \]
  where we have used that $\Z_j^* \Upsilon_j^U = 0$ by construction. 
  As $ \Z_j^*\U_j$ is invertible, we have that $P^U(A_j^{U,Z}) \circ (E_1R_B)=0$.
  Moreover,
  $R^U = \Upsilon_j^U R_B P_j^U\neq 0$ implies that $P_j^U \neq 0$, i.e., 
  $P^U(\lambda)$ is a block characteristic polynomial 
  for $A_j^{U,Z}$ with respect to  $E_1R_B$. 
\end{proof}

The above result characterizes $R^U$  in terms of  
$P^U(A) \circ B$,
where the latter is a block characteristic polynomial. However, a matrix can have several block characteristic polynomials related with 
the same block vector. For instance, given a
full rank block vector $W$, the set of all 
block characteristic polynomials for the zero 
matrix of size $js \times js$ associated with $W$ is made of all $s \times s$ matrix 
polynomials of degree $j$ such that 
$P_0 = 0$. This is different from the standard
characteristic polynomial, that is uniquely defined for all square matrices. 

In the next result 
we show that if the block vector is of the form $E_1 M$, for an invertible $s\times s$ matrix $M$, and 
the matrix is block upper Hessenberg  
with invertible subdiagonal blocks, we can identify all the degrees of freedom of the set of block characteristic polynomials. We recall that a matrix polynomial is \emph{regular} if it is 
square and its determinant is not identically zero over $\mathbb C$.

\begin{lemma}
  \label{lem:unique-block-characteristic-poly}
  Let $M\in\mathbb C^{s\times s}$ be invertible, $H_j$ be the block upper Hessenberg 
  matrix 
  \begin{equation} \label{eq:Hj1}
    H_j = \begin{bmatrix}
      \Phi_{1} & \Xi_{1,2}&\dots&\Xi_{1,j} \\ 
      \Gamma_2 & \Phi_2 & \ddots &\vdots \\ 
      & \ddots & \ddots & \Xi_{j-1,j} \\ 
      & & \Gamma_{j} & \Phi_j \\
    \end{bmatrix} \in \mathbb C^{js \times js}, 
  \end{equation}
  and let 
  $\mathcal S$ be the set of matrix polynomials satisfying 
  \[
  \mathcal S := \{ 
  P(\lambda) \in \mathbb C[\lambda]^{s \times s} \ | \ \deg(P)=j, \ 
  P(H_j) \circ (E_1M) = 0
  \}. 
  \]
  If the subdiagonal blocks $\Gamma_2,\dots,\Gamma_j,$ of $H_j$ 
  are invertible, then
  there exist $\Theta_i \in \mathbb C^{s \times s}$, $i=0,\dots,j-1$, 
  such that
  $\mathcal S$ can be characterized as follows:
  \[
  \mathcal S = \{ 
  P(\lambda) =(z^j I + z^{j-1} \Theta_{j-1} + 
  \ldots + z \Theta_1 + \Theta_0)  P_j \ | \ 
  P_j \in \mathbb C^{s \times s},\ P_j \neq 0
  \}. 
  \]	
  Moreover, $P(\lambda)\in\mathcal S$ is regular if and only if its leading coefficient $P_j$ is invertible.
\end{lemma}

\begin{proof}
  In view of Remark~\ref{rem:less-magic}, it is sufficient to show the claim in the case $M=I_s$.
  Note that the first column of a positive integer power 
  of 
  $H_j$ has the  form
  \[
  H_j^k E_1 = \left[\begin{smallmatrix}
    \times \\ 
    \vdots \\
    \times \\ 
    \Gamma_{k+1} \ldots \Gamma_2 \\ 
    0 \\ \vdots \\ 0 
  \end{smallmatrix}\right], \qquad 
  \forall\ k < j, 
  \]
  where only the first $k+1$ block entries can be nonzero. 
  Let $P(\lambda) = \sum_{i = 0}^j P_i z^i$; 
  the condition $P(H_j) \circ E_1 = 0$ can be written 
  explicitly as 
  \[
  H_j^j E_1 P_j + H_j^{j-1} E_1 P_{j-1} + 
  \ldots 
  + H_j E_1 P_1 + E_1 P_0 = 0. 
  \]
  Reading the last block row in the above equation yields 
  $
  \Gamma_{j} \ldots \Gamma_2 P_{j-1} = -(E_j^* H_j^j E_1)  P_j. 
  $
  Since $\Gamma_{j} \ldots \Gamma_2$ is invertible, $P_{j-1}$
  is uniquely determined from $P_j$, if the latter is given. Similarly, 
  reading the $k$th block rows yields relations of the form 
  \[
  \Gamma_{k+1} \ldots \Gamma_2 P_{k} = 
  - \sum_{i = k+1}^j (E_k^* H_j^i E_1) P_i, \qquad 
  k = 0, \ldots, j-1, 
  \]
  where the right-hand side 
  is a linear combination of the coefficients 
  $P_{k+1}, \ldots, P_{j}$. Therefore, $P_0, \ldots, P_{j-1}$ 
  are all uniquely determined once $P_j$ is fixed. Expanding 
  the relations shows that we can find matrices $M_i$ such that 
  $P_i = \Gamma_2^{-1} \ldots \Gamma_{k+1}^{-1} M_i P_j$.         
  In particular, if we set $\Theta_i = \Gamma_2^{-1} \ldots \Gamma_{k+1}^{-1} M_i$
  we obtain
  $P_i = \Theta_i P_j$. 
  
  Concerning the last part of the claim we recall that a matrix polynomial with invertible leading coefficient is always regular, and in the other case we have that $\mathrm{rank}(P(\lambda))\le \mathrm{rank}(P_j)$, for all $\lambda\in\mathbb C$.
\end{proof}

Lemma~\ref{lem:unique-block-characteristic-poly} says 
that for any choice of 
$P_j \neq 0$ there is a unique block characteristic polynomial with respect to $E_1M$ that has $P_j$ as leading coefficient. Moreover, Lemma~\ref{lem:unique-block-characteristic-poly} can be extended to
the set of block characteristic polynomials of $N$ with respect to $W$, whenever 
the pair $(N, W)$ is controllable. Before stating the result, we recall the definition of 
controllable pair. 

\begin{definition}
  Let $(N, W)$ be a pair of matrices with $N \in \mathbb C^{js \times js}$ and 
  $W \in \mathbb C^{js \times s}$; then, $(N, W)$ is a \emph{controllable 
  pair} if  the dimension of the vector subspace
  $
  \mathrm{colspan}\left(
    W, NW, \ldots, N^{j-1}W
  \right)$ is equal to $js$.
\end{definition}

We can 
always transform a controllable pair into a pair with a block upper Hessenberg matrix,
and $E_1$ as the block vector, and draw a connection between their block 
characteristic polynomials.

\begin{lemma}\label{lem:clenshaw-general} 
  Let $N\in\mathbb C^{js\times js}$, and $W\in\mathbb C^{js\times s}$ such that $(N,W)$ is a controllable pair, then
  \begin{itemize}
\item[$(i)$] There exists a unitary matrix $Q\in\mathbb C^{js\times js}$ such that
\begin{equation}\label{eq:hcontrollable}
Q^* NQ=H,\qquad Q^* W=E_1M,
\end{equation}
where $H$ is block upper Hessenberg, and $M$ is an invertible $s\times s$ matrix. 
\item[$(ii)$] All block upper Hessenberg matrices $H$ satisfying \eqref{eq:hcontrollable} for some $Q,M$, have invertible subdiagonal blocks.
\item[$(iii)$] $P(\lambda)$ is a block characteristic polynomial for 
  $H$ with respect to $E_1 M$ if and only if 
  $P(\lambda)$ is a block characteristic polynomial for $N$ with respect to $W$.	
  \end{itemize}
  \end{lemma}
\begin{proof}
A pair of matrices $Q$, and $H$ satisfying the first claim is obtained by running the block Arnoldi method on $N,W$ until completion, i.e. for $j$ steps. Property $(ii)$ follows from the invariance of controllability under change of basis, and the fact that a pair $(H,E_1M)$ is controllable if and only if $H$ has invertible subdiagonal blocks, \upd{ as we have already shown in the discussion after Assumptions~1}. Finally, we see that
\begin{align*}
P(H)\circ(E_1M)=0\ &\iff\ P(Q^* NQ)\circ(E_1M)=0\\
&\iff\ Q^*\left[P(N)\circ(QE_1M)\right]=0\\
&\iff\ P(N)\circ W=0.\qedhere
\end{align*}
\end{proof}
Lemma~\ref{lem:clenshaw-general} implies the following result that extends Lemma~\ref{lem:unique-block-characteristic-poly}.

\begin{corollary} \label{cor:block-characteristic-poly-controllable}
  Let $(N, W)$ be a controllable pair, with $N \in \mathbb C^{js \times js}$ 
  and $W \in \mathbb C^{js \times s}$. There 
  exist $\Theta_i \in \mathbb C^{s \times s}$, 
  for $i = 0, \ldots, j-1$, such that the 
  set $\mathcal S$ of all block characteristic polynomials 
  of $N$ with respect to $W$ can be written as 
  \[
  \mathcal S = \{ 
  P(\lambda) =(z^j I + z^{j-1} \Theta_{j-1} + 
  \ldots + z \Theta_1 + \Theta_0)  P_j \ | \ 
  P_j \in \mathbb C^{s \times s},\ P_j \neq 0
  \}. 
  \]	
  Moreover, $P(\lambda)\in\mathcal S$ is regular if and only if its leading coefficient $P_j$ is invertible.
\end{corollary}
Under the assumptions of Corollary~\ref{cor:block-characteristic-poly-controllable}
we indicate the block characteristic polynomial 
with $P_j = I$ as the \emph{monic block characteristic polynomial} of $N$ with 
respect to $W$.


\subsection{Residual and error block polynomials}\label{sec:error-poly}

We now have all the ingredients to express the 
residuals $\res_{B,j}(z)$ defined in \eqref{eq:resj}
in terms of the monic block characteristic 
polynomials of $A_j^{U,Z}$.
The condition $X_{B,j}(z) \in \mathcal 
K^\square_{j}(A, B)$  implies 
\[
X_{B,j}(z) = \chi_{B,j}(A, z) \circ B, 
\qquad 
\chi_{B,j}(\lambda, z) = P_{j-1}(z) \lambda^{j-1} + \ldots +
P_1(z) \lambda +  P_0(z), 
\]
where $\chi_{B,j}$ is a parameter dependent matrix polynomial of degree $j-1$. The residual $\res_{B,j}(z)$ belongs to $\mathcal K^\square_{j+1}(A, B)$ and 
satisfies a similar relation 
\[
\res_{B,j}(z) = \gamma_{B,j}(A, z) \circ B, 
\qquad 
\gamma_{B,j}(\lambda, z) = I_s - (z - \lambda) \chi_{B,j}(\lambda, z), 
\]
and, in particular, $\gamma_{B,j}(z, z) = I_s$.

The following \upd{lemma} generalizes the residual 
formula \eqref{eq:residual-s=1} to the block 
case.

\begin{lemma}\label{lem:res}
  Let $X_{B,j}(z)$, be the approximation 
  to $(zI - A)^{-1} B$  as in \eqref{eq:XBjarnoldi}.
  Then, under Assumptions 1,
  the residual $\res_{B,j}(z)$  defined 
  in \eqref{eq:resj} satisfies
  \[
  \res_{B,j}(z) = [\Lambda^U(A) \circ B]\Lambda^U(z)^{-1}, 
  \]
  and $\Lambda^U(z)$ is the monic block characteristic polynomial of $A_j^{U,Z}$ with respect to $E_1R_B$. 
\end{lemma}

\begin{proof}
  Since $\res_{B,j}(z)$ belongs to 
  $\mathcal K^\square_{j+1}(A, B) \cap 
  \mathcal Z^\perp$ we 
  can use Lemma~\ref{lem:ortho-poly} and write it as 
  $
  \res_{B,j}(z) = P^U(A, z) \circ B, 
  $
  where $P^U(\lambda, z)$ verifies 
  $P^U(A_j^{U,Z}, z) \circ (E_1R_B) = 0$. Then, by Lemma~\ref{lem:unique-block-characteristic-poly}, 
  we can write $P^U(\lambda, z) = \Lambda^U(\lambda) \rho(z)$, where $\Lambda^U(\lambda)$ is the monic block characteristic polynomial of $A_j^{U,Z}$ with respect to $E_1R_B$, and $\rho(z)$ is a 
  $s \times s$ matrix-valued function $\rho(z)$. 
  We remark that, following the same argument of 
  the proof of Lemma~\ref{lem:ortho-poly} we 
  may write 
  \begin{equation} \label{eq:upsilon-res}
    \res_{B,j}(z) = \Upsilon_j^U R_B \rho(z). 
  \end{equation}
  Since $\Upsilon_j^U R_B $ has full column rank, we can left 
  multiply by any left inverse, and obtain $\rho(z)$ 
  as a product of the residual (which is \upd{a holomorphic} function when $z$ is not an eigenvalue of $A_j^{U,Z}$), and a constant matrix. Hence, 
  $\rho(z)$ is holomorphic on the complement of the eigenvalues of $A_j^{U,Z}$. Imposing
  $P^U(\lambda, \lambda) = I_s$ provides the claim. 
\end{proof}

\subsection{Collinearity of residuals}
\label{sec:collinearity}

It is \upd{well known} that the residuals of shifted linear systems 
associated with FOM are collinear. This property extends to 
block FOM \cite{gutknecht,frommer2017block}, and to $\res_{B,j}(z)$ defined above 
obtained by imposing a Petrov-Galerkin condition.
To the best of our knowledge, \upd{the collinearity of residuals in the Petrov-Galerkin scenario has not been pointed 
out in the literature}, and this property follows from the results of the previous section. 
In particular, from  Equation~\eqref{eq:upsilon-res} 
and the definition of $\Upsilon_j^U$ in Lemma~\ref{lem:poly-lanczos}
 we have that
all residuals are block collinear (or cospatial, using the terminology 
from \cite{frommer2017block}) to the block vector 
$\Pi_{U,Z}U_{j+1}=(I - \U_j (\Z_j^* \U_j)^{-1} \Z_j^*)U_{j+1}$. The dependency on $z$
comes from the right multiplication with the inverse of $\Lambda^U(z)$. 
With the aim of providing a more explicit formula for the latter, 
we state two technical lemmas.

\begin{lemma} \label{lem:ritz-values}
  Let $H$ be a $js \times js$ block upper Hessenberg 
  matrix as in \eqref{eq:Hj1}, with 
  invertible subdiagonal blocks $\Gamma_2, \ldots, \Gamma_j$. 
  Let $\Lambda(z)$ be its 
  block characteristic polynomial with respect to 
  $E_1$ with leading coefficient 
  $\Gamma_2^{-1} \ldots \Gamma_j^{-1}$. 
  Then, $H$ is a linearization of the matrix polynomial 
  $\Lambda(z)$. In addition, for every simple eigenvalue $\theta$ 
  of $H$, with right and left eigenvectors 
  $v$ and $w$, we have that
  \[
    w_1^* \Lambda(\theta) = 0, \quad \text{ and } \quad 
    \Lambda(\theta) v_j = 0,		  
   \]
   where $v,w$ are partitioned in blocks of size $s$ as follows:
   \[
    v = \begin{bmatrix}
    v_1 \\
    \vdots \\
    v_{j-1} \\
    v_j
    \end{bmatrix}, \qquad 
    w = \begin{bmatrix}
    w_1 \\
    w_2 \\ 
    \vdots \\
    w_j
    \end{bmatrix}.
  \]
  Finally, it holds that 
  $
    w^* v = w_1^* \Lambda'(\theta) v_j. 
  $
\end{lemma}
The proof of the above result relies on a 
Clenshaw-like recurrence 
for the evaluation of a matrix polynomial, and is
postponed to
Section~\ref{sec:computational}.

\begin{remark} \label{rem:monic-eigenvectors}
  If one is interested in the eigenvectors associated with 
  an eigenvalue $\theta$ 
  of $\Lambda^U(z) M$, for an invertible matrix $M$, they can 
  be retrieved as $w_1$ and $M^{-1} v_j$. The scaling property 
  $w^* v = w_1^* (\Lambda(\theta) M)' M^{-1} v_j$ is preserved. This 
  can be used to compute the eigenvectors for the monic version of 
  $\Lambda^U(z)$. 
\end{remark}

\begin{lemma} \label{lem:keldysh-inverse}
  Let $\Lambda_{j}^U(z)$ be the monic block characteristic 
  polynomial of $A_j^{U,Z}$ with respect to $E_1R_B$. 
  Then
  \[
    \Lambda_{j}^U(z)^{-1} = 
      R_B^{-1} 
    (\Gamma_2^U)^{-1} \ldots (\Gamma_{j}^U)^{-1} E_j^* (zI - A_j^{U,Z})^{-1} E_1 R_B.
  \]
\end{lemma}

\begin{proof}
  In view of Remark~\ref{rem:less-magic},  $P(z) =  R_B^{-1} \Lambda_{j}^U(z) R_B(z) (\Gamma_2^U)^{-1} \ldots (\Gamma_{j}^U)^{-1}$ is 
  the block characteristic polynomial of $A_j^{U,Z}$ with respect to $E_1$, 
  with leading coefficient $(\Gamma_2^U)^{-1} \ldots (\Gamma_{j}^U)^{-1}$.
  Let $\theta_i$ be the eigenvalues 
  of $A_j^{U,Z}$, and note that they coincide with the eigenvalues 
  of $\Lambda_{j}^U(z)$, and of $P(z)$.
      Let us first assume that 
      all eigenvalues of $\Lambda^U_j(z)$
      are simple. 
      \upd{Since $P(z)^{-1}$ is 
      a rational matrix function with poles at the eigenvalues 
      of $P(z)$, it can be written as a partial fraction expansion
      over its poles. Keldysh's theorem \cite[Theorem~2.4]{glr}\footnote{See also \cite{guettelnonlinear} for a recent work dealing with Keldysh's theorem with a notation closer to the one used in this paper.}
      yields an explicit expression for such formula:}
  \begin{equation} \label{eq:keldysh-1}
    P(z)^{-1} = \sum_{i = 1}^{js} \frac{1}{z - \theta_i} v_i w_i^* 					
  \end{equation}
  where $v_i, w_i \in \mathbb C^s$ 
  are right and left eigenvectors 
  of $P(z)$
  associated with 
  $\theta_i$, normalized imposing 
  \upd{$w_i^* P'(\theta_i) v_i = 1$}. 
   Lemma~\ref{lem:ritz-values} ensures that
$v_i, w_i$ are determined by 
  $
    v_i = E_j^{*} M e_i,
    w_i = E_1^{*} M^{-T} e_i, 
  $
  where $M \in \mathbb C^{js \times js}$ 
  is any eigenvector matrix of $A_j^{U,Z}$. 
  Substituting the expression for the eigenvectors in 
  Equation~\eqref{eq:keldysh-1} yields 
  \[
    P(z)^{-1} = 
      E_j^* M 
      \left( \sum_{i = 1}^{js} \frac{e_i e_i^*}{z - \theta_i} \right) 
    M^{-1} E_1 = E_j^* (zI - A_j^{U,Z})^{-1} E_1.
  \]
  Then, the claim follows from $\Lambda_j^{U}(z)^{-1} = R_B^{-1} (\Gamma_2^U)^{-1} \ldots (\Gamma_j^U)^{-1} P(z)^{-1} R_B$. If there are non-trivial Jordan triples in 
      $\Lambda_j^U(z)$, the result is obtained by
      using the more general version of Keldysh's theorem
      \cite{glr,guettelnonlinear} and that $A_j^{U,Z}$ is
      a linearization of $\Lambda_j^U(z)$.
\end{proof}

Building on Lemma~\ref{lem:ritz-values} and 
Lemma~\ref{lem:keldysh-inverse}, we can  provide an 
explicit formula for the residual map $\res_{B,j}(z)$, in 
the Petrov-Galerkin case. 

\begin{theorem} \label{thm:res-H}
  Under the assumptions of Lemma~\ref{lem:res}, then: 
  \[
    \res_{B,j}(z) = \Pi_{U,Z}U_{j+1} 
    \Gamma_{j+1}^U 
     E_j^* (zI - A_j^{U,Z})^{-1} E_1R_B .
  \]
\end{theorem}
\begin{proof}
  The claim follows by writing 
  $\res_{B,j}(z) = \Upsilon_j^U R_B \Lambda^U(z)^{-1}$ as 
  in the proof of Lemma~\ref{lem:res}  (Equation~\eqref{eq:upsilon-res}), 
  and using Lemma~\ref{lem:keldysh-inverse} 
  for the inverse of $\Lambda^U(z)$.
\end{proof}
In the particular case of Galerkin projections, all the residuals 
are block collinear to the last block vector $U_{j+1}$, generated by 
the block Arnoldi procedure, and we recover the well-known 
expression  
$
  \res_{B,j}(z) = U_{j+1} 
  \Gamma_{j+1}^U 
   E_j^* (zI - A_{j}^{U})^{-1} E_1R_B
$ \cite{frommer2017block}.

\subsection{Block rational Krylov methods}

Let us generalize the results of Section~\ref{sec:error-poly}, 
and Section~\ref{sec:collinearity} 
to the case of rational Krylov subspaces. Given a set of shift parameters $\Sigma=\{\sigma_1,\dots, \sigma_{j-1}\}$, we introduce the block rational Krylov subspace
\begin{align*}
  \mathcal{RK}^\square_j(A, B, \Sigma):=&\bspan\{B, (\sigma_1I-A)^{-1}B,\dots,(\sigma_{j-1}I-A)^{-1}B\},
\end{align*} 
and the polynomial
$$
\varphi(z):= \prod_{\begin{smallmatrix}i=1\\ \sigma_i\ne \infty\end{smallmatrix}}^{j-1}(z-\sigma_i).$$	
Note that, the above subspace can be rewritten by means of the $\circ$ operator as
\begin{align*}
  \mathcal{RK}^\square_j(A, B, \Sigma)=&\{\varphi(A)^{-1} P(A)\circ B:\ P(\lambda)\in\mathbb C[\lambda]^{s\times s},\ \mathrm{deg}\ P(\lambda)< j\}.
\end{align*} 
When we write as subscript of the $\mathcal{RK}^\square$ symbol the cardinality of the set of shift parameters plus two, then we mean that two shifts are taken at infinity: 
\begin{align*}
  \mathcal{RK}^\square_{j+1}(A, B, \Sigma)&:=\mathrm{blockspan}\{B, AB, (\sigma_1I-A)^{-1}B,\dots,(\sigma_{j-1}I-A)^{-1}B\}\\
  &\ =\{\varphi(A)^{-1} P(A)\circ B:\ P(\lambda)\in\mathbb C[\lambda]^{s\times s},\ \mathrm{deg}\ P(\lambda)\le j\}.
\end{align*} 
In the following, we denote by $\U_j$ the block unitary basis of $\mathcal {RK}^\square_{j}(A, B, \Sigma),$ obtained with the rational block Arnoldi algorithm \cite{elsworth20}, and by $\Z_j$ a unitary basis of a $js$-dimensional subspace 
  of $\mathbb C^{n}$.    
  We now introduce
the Block Rational Arnoldi Decomposition (BRAD) 
  \cite{elsworth20}, 
  that is returned by the block rational Arnoldi procedure for 
building $\U_j$. A BRAD consists in a pair of
   $(j+1)s \times js$ 
block upper triangular matrices 
$\hbradu_{U,j}, \kbradu_{U,j}$ that satisfy
\begin{equation}\label{eq:rational-arnoldi-relation}
  A \U_{j+1} \kbradu_{U,j} = \U_{j+1} \hbradu_{U,j}. 
\end{equation}
In analogy to the polynomial case, we denote by $\hbrad_{U,j}$ and 
$\kbrad_{U,j}$ the $js \times js$ block upper Hessenberg 
matrices obtained by extracting the top $js$ rows 
from $\hbradu_{U,j}$ and $\kbradu_{U,j}$, 
respectively. Moreover, we keep the notation involving 
  $\Gamma_j^U, \Xi_{i,j}^U, \Phi_j^U$ to denote the blocks 
  of $\hbradu_{U,j}$. 
  Similarly to Section~2, we identify a set of general assumptions that will be 
used throughout this section.

\medskip 
\noindent \fbox{
  \parbox{.95\linewidth}{
    \begin{center}
      \textbf{Assumptions 2}
    \end{center}
\begin{enumerate}
\item The block rational Arnoldi procedure 
for building a block orthogonal basis of $\mathcal{RK}^\square_{j+1}(A,B,\Sigma)$
has been run
without encountering breakdown nor deflations, and the 
last pole is chosen at infinity. 
\item The block unitary bases $\Z_j, \U_j \in \mathbb{C}^{n \times js}$ 
are such that $\Z_j^* \U_j$ is invertible.\label{ass:ZjUj}
\item The shifts $\sigma_i$ are not eigenvalues of $A_j^{U,Z}$. 
\end{enumerate}}}
\medskip 
  
  Note that, in general, the matrix $A_j^U = \U_j^* A \U_j \neq \hbrad_{U,j}$, and both $A_j^U$ 
  and $A_j^{U,Z}=(\Z_j^*\U_j)^{-1}\Z_j^*A\U_j$ are not block upper Hessenberg.
  We remark that the order of the shift parameters used in the 
  rational block Arnoldi procedure is encoded in 
  the matrices $\hbradu_{U,j}$ and $\kbradu_{U,j}$; in particular, when 
  the last shift is chosen at infinity, the last block row 
  of $\kbradu_{U,j}$ is zero. We remark that, even if the last 
  shift is not infinity, we can always reduce to this case by
  a (cheap) orthogonal transformation that reorders the basis, see 
  \cite[Section~4]{berljafa2015generalized}, and 
  \cite[Section~4.1]{casulli2024efficient}.

To extend the results concerning the error and residual block polynomials to the rational Krylov setting, we rely on the well known relation:
\begin{align}
  \mathcal{RK}^\square_{j+1}(A, B, \Sigma)= \mathcal{K}^\square_{j+1}(A,\varphi(A)^{-1}B)\label{eq:k1}.
\end{align}
We remark that Assumptions~2 imply that we can run the block Arnoldi method with $A$  
and starting block vector $\varphi(A)^{-1}B$, for up to $j+1$ 
steps without breakdowns
nor deflations. Denote by 
$\V_j$ the unitary block
basis of $\mathcal K^\square_{j}(A, \varphi(A)^{-1}B),$  obtained via the block Arnoldi procedure, 
and let $A_{j}^{V,Z}$ be the corresponding block upper Hessenberg matrix. Finally, indicate by $S\in\mathbb C^{js\times js}$ the unitary matrix  such that $\U_j=\V_jS$. 

\upd{
We now prove the following generalizations of Lemma~\ref{lem:poly-lanczos}, and Lemma~\ref{lem:ortho-poly}. 
\begin{lemma}\label{lem:rat-lanczos}
	Let $P(\lambda)=\lambda^kP_k+\dots + \lambda P_1+P_0$ be an $s\times s$ matrix polynomial of degree $k$. Under Assumptions~\ref{ass:ZjUj}:
	$$
	\varphi(A)^{-1}P(A)\circ B=\begin{cases}
		\U_j\left[\varphi(A_j^{U,Z})^{-1}P(A_j^{U,Z})\circ(E_1 R_B)\right]& j>k\\[.2cm]
		\U_j\left[\varphi(A_j^{U,Z})^{-1}P(A_j^{U,Z})\circ(E_1 R_B)\right] + \Theta_j^UR_B P_j& j=k,
	\end{cases}
	$$
	where $\Theta_j^U= \Pi_{U,Z}U_{j+1}\Gamma_{j+1}^UM_j^U$, for a certain $M_j^U\in\mathbb C^{s\times s}$. If all poles in $\Sigma$ are different from $\infty$, then $M_j^U=E_j^TK_{U,j}^{-1}E_1$.
\end{lemma}	
\begin{proof}
	Let us denote by $\V_j$ the block basis generated via the Arnoldi procedure with starting block vector $\varphi(A)^{-1}B$; since the $(j+1)$th pole is chosen at infinity, without loss of generality, we assume that $U_{j+1}=V_{j+1}$. Moreover, let $S\in\mathbb C^{sj\times sj}$ be the orthogonal matrix such that $\V_j S =\U_j$. In addition, since $B = \varphi(A)\varphi(A)^{-1}B\in \mathcal K^\square_{j}(A,\varphi(A)^{-1}B)$, 
	in view of Lemma~\ref{lem:poly-lanczos} (using 
	$\varphi(\lambda)I_s$ as matrix polynomial) we have 
	\begin{align}\label{eq:varphi(h)}
		B&=\V_j\varphi(A_{j}^{V,Z})E_1V_1^*\varphi(A)^{-1}B = 	\U_j\varphi(A_j^{U,Z})S^{-1}E_1 V_1^*\varphi(A)^{-1}B\nonumber\\
		&\Longrightarrow S^{-1}E_1V_1^*\varphi(A)^{-1}B=\varphi(A_j^{U,Z})^{-1}\U_j^*B.
	\end{align}

	Then, by applying Lemma~\ref{lem:poly-lanczos} we have
	\begin{align*}
		\varphi(A)^{-1}&P(A)\circ B=P(A)\circ (\varphi(A)^{-1}B)\\ &= \V_j\left[P(A_j^{V,Z})\circ(E_1V_1^*\varphi(A)^{-1}B)\right] + \Pi_{VZ}V_{j+1}\Gamma_{j+1}^V\dots\Gamma_2^VE_1V_1^*\varphi(A)^{-1}BP_j,
	\end{align*}
	where  $P_j=0$ in the case $k<j$. Concerning the first addend we have that
	\begin{align*}
		\V_j P(A_j^{V,Z})\circ E_1V_1^*\varphi(A)^{-1}B
		&=\V_j S P(S^{-1}A_j^{V,Z}S)\circ S^{-1}E_1V_1^*\varphi(A)^{-1}B\\
		&=\U_jP(A_j^{U,Z})\circ \varphi(A_j^{U,Z})^{-1}E_1R_B,
	\end{align*}
	where in the last equality we used equation~\eqref{eq:varphi(h)}. 
	
	Note that $\Pi_{VZ}=\Pi_{U,Z}$, 
	 $V_{j+1}=U_{j+1}$, and $\Gamma_{j+1}^U=\Gamma_{j+1}^V$; since $B=U_1R_B$, setting $M_j^U=\Gamma_{j}^V\dots\Gamma_2^VE_1V_1^*\varphi(A)^{-1}U_1$ we get the first part of the second claim.
	 
	Finally, we show the identity $\Theta_j^U=\Pi_{VZ}V_{j+1}\Gamma_{j+1}^V\dots\Gamma_2^VE_1V_1^*\varphi(A)^{-1}BP_j$, under the assumptions that there are no poles at infinity in $\Sigma$.  The claim follows from
	\begin{align*}
		\Gamma_{j+1}^UE_j^TK_{U,j}^{-1}E_1 R_B &=U_{j+1}^*A\U_j E_1R_B=U_{j+1}^*AB 
		=U_{j+1}^*A\varphi(A)\varphi(A)^{-1}B 
    \\&=V_{j+1}^* A\varphi(A)\V_j\V_j^* \varphi(A)^{-1}B
		=\Gamma_{j+1}^V\dots \Gamma_{2}^VV_1^*\varphi(A)^{-1}B,
	\end{align*}
  where we applied 
  Lemma~\ref{lem:poly-lanczos} using
  that $\lambda\varphi(\lambda)$ is a monic polynomial
  of degree $j$.
\end{proof}}

\begin{lemma} \label{lem:ortho-poly-rat}
  Let $R^U \in \mathcal {RK}^\square_{j+1}(A, B, \Sigma) \cap \mathcal Z^\perp$ be a nonzero block vector. Under Assumptions 2, there exists a block characteristic polynomial $P^U(\lambda)$
  for $A_j^{U,Z}$ with respect to $E_1R_B$ such that
  $R^U = \varphi(A)^{-1}P^U(A) \circ B$. 
\end{lemma}
\upd{\begin{proof}
 By definition of $\mathcal{RK}_{j+1}^\square(A,B,\Sigma)$, $R^U=\varphi(A)^{-1}P(A)\circ B$ for a certain $s\times s$ matrix polynomial $P(\lambda)$ of degree $j$. In view of Lemma~\ref{lem:rat-lanczos}, we can write
 $$
 R^U= \U_j \varphi(A_{j}^{U,Z})^{-1}P(A_j^{U,Z})\circ E_1R_B+ \Pi_{U,Z}U_{j+1}\Gamma_{j+1}^UM_j^UR_BP_j.
 $$	
 The condition $R^U\in\mathcal Z^\perp$ implies 
 $
 0=\Z_j^* R^U=\Z_j^*\U_j\varphi(A_{j}^{U,Z})^{-1}P(A_j^{U,Z})\circ E_1R_B,
 $
and since $\Z_j^*\U_j\varphi(A_j^{U,Z})^{-1}$ is an invertible	 matrix, this gives the claim.
\end{proof}}
As said,  $A_j^{U,Z}$ is not block upper Hessenberg in general. On the other hand, a matrix polynomial is a block characteristic polynomial for $A_j^{U,Z}$ with respect to $E_1 R_B$ if and only if it is a block characteristic polynomial for $\widehat H_{V,j}$ with respect to $E_1 V_1^* \varphi(A)^{-1}B$. Therefore, in view of Lemma~\ref{lem:unique-block-characteristic-poly}, we can claim that the 
monic block characteristic polynomials of $A_j^{U,Z}$ with respect to 
$E_1 R_B$ is unique, and
all the other block characteristic polynomials can be obtained by multiplying it on the right with an $s\times s$ matrix which will constitute the leading coefficient.

Similarly to the discussion at the beginning of Section~\ref{sec:error-poly}, we have that under Assumptions 2, the Petrov-Galerkin approximate solution satisfies:
\begin{equation}\label{eq:xbj-rat}
  \begin{split}
    X_B(z)\approx X_{B,j}(z)&:=\U_j (zI-A_j^{U,Z})^{-1}E_1 R_B=\varphi(A)^{-1}\chi_{B,j}(A, z)\circ B ,
  \end{split}
\end{equation}
for a certain $z$-dependent matrix polynomial $\chi_{B,j}(A, z)$ of degree at most $j-1$. The residual is then given by
\begin{align*}
  \res_{B,j}(z) = \gamma_{B,j}(A, z) \circ B, 
  \qquad 
  \gamma_{B,j}(\lambda, z) = I_s - (z - \lambda) \frac{\chi_{B,j}(\lambda, z)}{\varphi(\lambda)},
\end{align*}
so that, $\gamma_{B,j}(z, z) = I_s$. By means of Lemma~\ref{lem:ortho-poly-rat}, we have that $\gamma_{B,j}(A, z) \circ B= \varphi(A)^{-1}P^U(A, z)\circ B$ where $P^U(A, z)$ is  a block characteristic polynomial of $A_j^{U,Z}$ with respect to $E_1R_B$. Combining all previous observations leads to the following result.
\begin{theorem}\label{thm:residual-rat}
  Let $X_{B,j}(z)$ be the Petrov-Galerkin approximation defined in \eqref{eq:xbj-rat}, 
  and $\res_{B,j}(z)$  be the associated residual. Under Assumptions 2, we have 
  \begin{align}\label{eq:residual-rat}
    \res_{B,j}(z) &= \varphi(z)\varphi(A)^{-1}\left(\Lambda^U(A) \circ B\right) \Lambda^U(z)^{-1}, 
  \end{align}
  and $\Lambda^U(\lambda)$ is the monic block characteristic polynomial of $A_j^{U,Z}$ with respect to $E_1R_B$. 
\end{theorem}
\begin{proof}
  The proof follows the same argument of the one of Lemma~\ref{lem:res}. In view of Lemma~\ref{lem:ortho-poly-rat} we can write
  $$
  \res_{B,j}(z) =\varphi(A)^{-1}P^U(A,z)\circ B,
  $$
  where the dependency on $z$ is only in $P^U$. By rewriting $P^U(\lambda, z)= \Lambda^U(\lambda)\rho(z)$ and imposing the condition $P^U(\lambda, \lambda)=\varphi(\lambda)$ we get the claim. 
\end{proof}
\upd{Finally, we provide a useful expression for the quantity $\phi(A)^{-1}\left(\Lambda^U(A)\circ B\right)$. The proof requires some technical intermediate results, and it is deferred to Section~\ref{sec:computational}.
\begin{corollary}\label{cor:final-formula}
	Under the assumptions of Theorem~\ref{thm:residual-rat},
  then
	$
	\varphi(A)^{-1}\Lambda^U(A)\circ B = \Pi_{U,Z} U_{j+1} \Gamma_{j+1}^U R_B.
	$	
\end{corollary}
}

\subsection{Collinearity of residuals in the rational case}
\label{sec:collinearity-rat}

The collinearity property of Section~\ref{sec:collinearity} extends 
to the rational Krylov case. 
%
If the last 
shift is at infinity we have that 
$A_j^U = \U_j^* A \U_j = \hbrad_{U,j} \kbrad_{U,j}^{-1}$, and we can extend Theorem~\ref{thm:res-H}  as follows.
\begin{theorem} \label{thm:res-H-rat}
  Under the assumptions of Theorem~\ref{thm:residual-rat} we have
  \begin{equation}\label{eq:res-H-rat}
  \res_{B,j}(z) = \Pi_{U,Z}U_{j+1} 
  \Gamma_{j+1}^U 
  E_j^* K_{U,j}^{-1} (zI - A_j^{U,Z})^{-1} E_1R_B,
  \end{equation}
  where $\Gamma_{j+1}^U$ denotes  the $(j+1,j)$ block of $\hbradu_{U,j}$.
\end{theorem}
\begin{proof}
In view of Theorem~\ref{thm:residual-rat}, we have that the residual 
can be written in the form 
$\res_{B,j}(z) =\varphi(z) [\Lambda^U(A) \circ \varphi(A)^{-1} B] \Lambda^U(z)^{-1}$. 
As already shown in the proof of Lemma~\ref{lem:ortho-poly-rat}, $\Lambda^U(z)$ is also the monic block characteristic
polynomial of $A_{j}^{V,Z}$ with respect to $E_1V_1^*\varphi(A)^{-1}B$.

Therefore, 
we have that $[\Lambda^U(A) \circ \varphi(A)^{-1} B] \Lambda^U(z)^{-1}$ is the
Petrov-Galerkin residual for the linear system $(zI - A) X = \varphi(A)^{-1}B$ 
associated with the polynomial Krylov subspace approximation, which according to Theorem~\ref{thm:res-H} 
can be written as 
\[
(I - \U_j (\Z_j^* \U_j)^{-1} \Z_j^*)V_{j+1} 
\Gamma_{j+1}^V 
 E_j^* (zI - \widehat H_{V,j})^{-1} \V_j^*\varphi(A)^{-1}B, 
\]
where $\Gamma_{j+1}^V$ is the $(j+1,j)$-block of 
$\underline H_{V,j}$, and we used that $(I - \U_j (\Z_j^* \U_j)^{-1} \Z_j^*) = 
(I - \V_j (\Z_j^* \V_j)^{-1} \Z_j^*)$ as the projector does not 
depend on the choice of basis. Since $\U_j$ and $\V_j$ span the same 
subspace, and ---since the last pole 
is $\infty$--- the same holds for $U_{j+1}$ and $V_{j+1}$, we have the existence of unitary matrices $S, T$ such that
$\U_j = \V_j S$, and $U_{j+1} = V_{j+1} T$. Further,  $A_{j}^{U,Z} = S^{-1} A_{j}^{V,Z} S$, and 
$\Gamma_{j+1}^U = T^{-1} \Gamma_{j+1}^V$.
This implies 
\begin{align*}
\res_{B,j}(z) &= \varphi(z) \Pi_{U,Z}V_{j+1} 
\Gamma_{j+1}^V 
 E_j^* (zI - A_{j}^{V,Z})^{-1} \V_j^*\varphi(A)^{-1}B \\ 
 &= \varphi(z) \Pi_{U,Z} V_{j+1} 
 E_{j+1}^* \underline{H}_{V,j} S (zI - A_j^{U,Z})^{-1} S^{-1} \V_j^* \varphi(A)^{-1}B \\
&= \varphi(z) \Pi_{U,Z} U_{j+1} 
E_{j+1}^* \hbradu_{U,j} K_{U,j}^{-1} (zI - A_j^{U,Z})^{-1} \varphi(A_j^{U,Z})^{-1} E_1 R_B \\
 &= \Pi_{U,Z} U_{j+1} 
\Gamma_{j+1}^U E_{j}^* K_{U,j}^{-1} \varphi(z)(zI - A_j^{U,Z})^{-1} \varphi(A_j^{U,Z})^{-1} E_1 R_B, 
\end{align*}
where we have used the identity in \eqref{eq:varphi(h)}.

We now assume without loss of generality that all $\sigma_i \neq \infty$ for $i = 1, \ldots, j - 1$. By means of the resolvent properties, we can write 
\begin{align*}
\varphi(z)(zI - A_j^{U,Z})^{-1} &\varphi(A_j^{U,Z})^{-1} =\varphi(z)(zI - A_j^{U,Z})^{-1} (A_j^{U,Z} -\sigma_1 I)^{-1}\prod_{i=2}^{j-1}(A_j^{U,Z} -\sigma_i I)^{-1} \\
 &=[(zI-A_j^{U,Z})^{-1}-(A_j^{U,Z} -\sigma_1 I)^{-1}]\prod_{i=2}^{j-1}(z-\sigma_i)(A_j^{U,Z} -\sigma_i I)^{-1}.
\end{align*}
By induction
$
\varphi(z)(zI - A_j^{U,Z})^{-1} \varphi(A_j^{U,Z})^{-1}=(zI - A_j^{U,Z})^{-1}-\sum_{i=1}^{j-2}\varphi_{i+1}(z)\varphi_{i}(A_j^{U,Z})^{-1},
$
with $\varphi_i(z):=\prod_{h=i}^{j-1}(z-\sigma_h)$. Note that $E_j^* K_{U,j}^{-1}=U_{j+1}^*A\U_j$, and $\varphi_{i}(A_j^{U,Z})^{-1}E_1R_B$ corresponds to the projection of $\varphi_{i}(A)^{-1}B$ onto $\bspan(\U_j)$, for $i=1,\dots,j-1$, because $\varphi_{i}(A)^{-1}B$ is contained in $\bspan(\U_{j})$. Therefore,
\begin{align*}
E_j^* K_{U,j}^{-1}\varphi_{i}(A_j^{U,Z})^{-1}E_1R_B &= U_{j+1}^*A \U_j\U_j^* \varphi_{i}(A)^{-1}B
=U_{j+1}^*A \varphi_{i}(A)^{-1}B=0,
\end{align*}
where the last equality follows from the fact that $A \varphi_{i}(A)^{-1}B\in\bspan(\U_j)$, by definition. In particular, the claim follows from 
$$E_j^* K_{U,j}^{-1}\left(\sum_{i=1}^{j-2}\varphi_{i+1}(z)\varphi_{i}(A_j^{U,Z})^{-1}\right)E_1R_B=0,
\qquad \forall z\in\mathbb C.$$
\end{proof}
\begin{remark}
In the Galerkin case, the formula in the statement of Theorem~\ref{thm:res-H-rat} boils down to
$
\res_{B,j}(z) = U_{j+1} 
\Gamma_{j+1}^U 
E_j^* K_{U,j}^{-1} (zI -  A_j^U)^{-1} E_1R_B.
$
\end{remark}

\subsection{A formula for the moment matching approximation error of transfer functions}

In this section we show  how to employ Theorem~\ref{thm:residual-rat} to retrieve closed formulas for the approximation error
of the transfer function of a linear time independent linear system (LTI). 
More precisely, we consider an LTI of the form 
\begin{equation} \label{eq:lti-ss}
  \begin{cases}
    x'(t) = Ax(t) + Bu(t) \\ 
    y(t) = Cx(t)
  \end{cases},\qquad x:\mathbb R \rightarrow \mathbb R^n,\qquad u:\mathbb R \rightarrow \mathbb R^s,\qquad C\in\mathbb C^{n\times s}.
\end{equation}
In this context, 
a crucial role is played by the  \emph{transfer function}, defined as
\begin{equation}\label{eq:trasfer}
  G(z):= C^*(zI-A)^{-1}B.
\end{equation}
The latter is an $s\times s$ matrix-valued rational function describing the input/output relation of the system in the frequency domain; more precisely we have $Y(z) = G(z)\cdot U(z)$, where $Y(z)$, and $U(z)$ are the Laplace transforms of $y(t)$, and $u(t)$, respectively~\cite{antoulas2001approximation}.

 Many model order reduction approaches for \eqref{eq:lti-ss} consist in designing reduced order models whose transfer function approximate \eqref{eq:trasfer} on a domain of interest. For instance, in the moment matching method one  considers a pair of rational block Krylov subspaces  $\mathcal{RK}^\square_j(A, B, \Sigma)$, and $\mathcal{RK}^\square_j(A^*, C, \Psi)$ ---with $\Psi:=\{\psi_1,\dots,\psi_{j-1}\}$---  and extract an approximation of \eqref{eq:trasfer} of the form 
$$
\widetilde G(z):= C^*X_{B,j}(z)=X_{C,j}(z)^*B,
$$
where $X_{B,j}(z),X_{C,j}(z)$ are the Petrov-Galerkin approximations, 
of the linear systems $(zI-A)X=B$, and $(zI-A^*)X=C$,  with respect to the aforementioned subspaces. 

When using rational Krylov subspaces, the rational matrix-valued function
$\widetilde G(z)$ interpolates $G(z)$ at the shifts, i.e., for $z \in\Sigma\cup\Psi$.
When the subspaces are polynomial Krylov subspaces, \upd{the Markov parameters are matched.\footnote{\upd{The Markov parameters of the transfer function $G(z)$ are the coefficients of the Taylor expansion of $G(z)$ at infinity \cite[Chapter 11]{antoulas2001approximation}.}}}

Let us provide an explicit error formula
for the interpolation error of the moment matching approach.
From the definitions of $X_{B,j}(z),\res_{B,j}(z),X_{C,j}(z)$, and $\res_{C,j}(z)$ we easily get the relations:
\begin{align*}
  C^*(zI-A)^{-1}B-C^*X_{B,j}(z)&=C^*(zI-A)^{-1}\res_{B,j}(z),\\
  C^*(zI-A)^{-1}&=\res_{C,j}(z)^*(zI-A)^{-1}+X_{C,j}(z)^*,
\end{align*}
so that the approximation error satisfies
\begin{equation*}
  \begin{split}
    G(z)-\widetilde G(z)&= C^*(zI-A)^{-1}B-C^*X_{B,j}(z)= C^*(zI-A)^{-1}\res_{B,j}(z)\\&=(\res_{C,j}(z)^*(zI-A)^{-1}+X_{C,j}(z)^*)\res_{B,j}(z) \\
    &=\res_{C,j}(z)^*(zI-A)^{-1}\res_{B,j}(z).
  \end{split}
\end{equation*}
Since $\widetilde G(z)$ only depends on the chosen pairs of subspaces, 
and is invariant with respect to the specific choice of bases, by applying our results on the expression of the residuals, we get the following formula for $G(z)-\widetilde G(z)$.
\begin{theorem}\label{thm:transfer}
  Under Assumptions 2, the error $E(z):= G(z)-\widetilde G(z)$ verifies
  \begin{align}\label{eq:remainder-transfer}
    E(z)&=\tau(z)\Lambda^Z(z)^{-*}[{\Lambda^Z}(A^*) 
    \circ C]^*\Theta(A, z)[\Lambda^U(A) \circ B] \Lambda^U(z)^{-1},
  \end{align}
  where $$
  \Theta(A, z) := (zI-A)^{-1}\tau(A)^{-1},
  \qquad 
  \tau(z):=\prod_{\psi_i\ne \infty} (z-\psi_i)\cdot\prod_{\sigma_h\ne \infty}(z-\sigma_h),
  $$
 $\Lambda^U(z)$ and 
 $\Lambda^Z(z)$ are 
 the monic block characteristic polynomials of $A_j^{U,Z}$ and $(A_j^{Z,U})^*$ with respect to $E_1R_B$, and $E_1R_C$.
Moreover, an equivalent characterization of the error function is
\begin{align}\label{eq:remainder-transfer2}
  E(z)&=R_C^*E_1^*(zI - (A^*)_j^{Z,U})^{-*} K_{Z,j}^{-*}E_j ( \Gamma_{j+1}^Z)^*Z_{j+1} \Pi_{U,Z}\nonumber\\ &\cdot (zI-A)^{-1}\Pi_{U,Z}U_{j+1} 
   \Gamma_{j+1}^U 
  E_j^* K_{U,j}^{-1} (zI - A_j^{U,Z})^{-1} E_1R_B.
\end{align}
  \end{theorem}

The above \upd{formulas} are interesting as they can provide some insights about the 
convergence of the approximation; however, we remark that their direct 
evaluation  would be as costly as 
explicitly computing $G(z) - \widetilde G(z)$.  Surrogate of the error functions might be obtained by combining \eqref{eq:remainder-transfer} and/or \eqref{eq:remainder-transfer2} 
with some heuristics that avoid the evaluation of the resolvent in the core of the two formulas; for instance, see \cite[Section 4]{lin2021transfer} for the single input single output case.

\section{Matrix functions approximation}
\label{sec:matfun}

This section is devoted to providing novel error formulas, and computable a posteriori error bounds for the 
approximation of $f(A)B$ by means of a block rational Krylov subspace.
We remark that when $s = 1$ the Galerkin approximation of $f(A)b$ onto a rational Krylov method, 
yields $g(A)b$ where $g(z)$ is a rational function interpolating
$f(z)$ at the Ritz values \cite[Theorem~3.3]{guttel2013rational}. 
The argument does not extend 
trivially to the case $s > 1$.
\subsection{Interpolation error formulas for matrix functions}

Let us consider the Petrov-Galerkin approximation $F_j \approx f(A)B$, with respect to the $\bspan$ of $\U_j$ and $\Z_j$,  
 defined by
$
  F_j:= \U_jf(A_j^{U,Z})E_1R_B. 
$

Using the Cauchy integral representations of $f(A)B$, and Theorem~\ref{thm:residual-rat} we get
\begin{align}
f(A)B-F_j&=\frac 1{2\pi\mathbf i}\int_{\partial \Omega}\left[(zI-A)^{-1}B-\U_j(zI-A_j^{U,Z})^{-1}E_1R_B\right]f(z)dz\nonumber\\
&=\frac 1{2\pi\mathbf i}\int_{\partial \Omega}(zI-A)^{-1}\res_{B,j}(z) f(z)dz \label{eq:f(a)res}
\end{align}
that involves the residual map.
It is natural to replace $\res_{B,j}(z)$ by either \eqref{eq:residual-rat} or \eqref{eq:res-H-rat}. By using the former equation we get the following result.
\begin{corollary} \label{cor:f-res}
Under the assumptions of Theorem~\ref{thm:residual-rat}, if 
$\Lambda^U(z)$ has simple and finite eigenvalues then 
\begin{align}\label{eq:cauchy-int}
  f(A) B - F_j &= \sum_{i=1}^{ds} \left[
  \varphi(A) f(A)   - \varphi(\theta_i) f(\theta_i) I \right] 
  (A - \theta_i I)^{-1} 
  \varphi(A)^{-1} (\Lambda^U(A) \circ B) v_i w_i^* \nonumber \\ 
  &\upd{= \sum_{i=1}^{ds} \left[
  \varphi(A) f(A)   - \varphi(\theta_i) f(\theta_i) I \right] 
  (A - \theta_i I)^{-1} 
  \Pi_{U,Z}U_{j+1}\Gamma_{j+1}^UR_B v_i w_i^*}
\end{align}
where $v_i, w_i$ are right and left eigenvectors of $\Lambda^U(z)$
associated with 
$\theta_i$, normalized such that $w_i^* (\Lambda^U)'(\theta_i) v_i = 1$.
\end{corollary}
\begin{proof}
Let $\Omega$ be a complex domain enclosing the spectra of $A$ and of 
$P(z)$, such that $\partial \Omega$ 
is a finite union of rectifiable Jordan curves.
Plugging \eqref{eq:residual-rat} in \eqref{eq:f(a)res} yields:
$$
f(A) B - F_j=\frac 1{2\pi\mathbf i}\int_{\partial \Omega}(zI-A)^{-1}\varphi(A)^{-1}(\Lambda^U(A)\circ B)\Lambda^U(z)^{-1}\varphi(z)f(z)dz.
$$
As in the proof of Lemma~\ref{lem:keldysh-inverse}, we write 
$
\Lambda^U(z)^{-1} = \sum_{i = 1}^{ds} \frac{1}{z - \theta_i}
v_i w_i^* 
$
where $v_i, w_i$ are right and left eigenvectors associated with 
$\theta_i$, normalized as in the statement. 
Replacing the latter 
in the integral formula, and denoting by $g(z):=f(z)\varphi(z)$, yields 
\begin{align*}
f(A)B-F_j&=\frac{1}{2 \pi i} \int_{\Gamma} (zI - A)^{-1} g(z) \varphi(A)^{-1}\left(\Lambda^U(A)\circ B\right) \Lambda^U(z)^{-1} dz 
\\ &=
\sum_{i=1}^{ds} \frac{1}{2 \pi i} \int_{\Gamma} (zI - A)^{-1} g(z) \varphi(A)^{-1}\left(\Lambda^U(A)\circ B\right) \frac{v_i w_i^*}{z - \theta_i}
dz.
\end{align*}
The function $g(z) / (z - \theta_i)$ has a pole at $z = \theta_i$.
By applying 
\cite[Theorem~4.1]{decaybounds}, that is a generalized version 
of Cauchy's integral formula for matrix functions, we obtain 
\begin{align*}
f(A)B-F_j&=\sum_{i=1}^{ds} \frac{1}{2 \pi i} \int_{\Gamma} (zI - A)^{-1} g(z) \varphi(A)^{-1}\left(\Lambda^U(A)\circ B\right) \frac{v_i w_i^*}{z - \theta_i}
dz \\ &= \sum_{i=1}^{ds} \left[g(A) (A - \theta_i I)^{-1}  + g(\theta_i) (\theta_i I - A)^{-1}  \right]\varphi(A)^{-1}\left(\Lambda^U(A)\circ B\right) v_i w_i^* \\
&= \sum_{i=1}^{ds} \left[g(A)   - g(\theta_i) I \right] (A - \theta_i I)^{-1} \varphi(A)^{-1}\left(\Lambda^U(A)\circ B\right) v_i w_i^*.
\end{align*}
\upd{The claim follows by applying Corollary~\ref{cor:final-formula}.}	
\end{proof}

\begin{remark} A consequence of Lemma~\ref{lem:ritz-values} is that if 
a block upper Hessenberg matrix $H$
 has all simple and distinct eigenvalues, then 
all eigenvalues of its block characteristic polynomial $\Lambda(z)$ are 
also simple and both the eigenvalues and 
eigenvectors can be readily computed from the Schur form 
of $H$, without the need of explicitly determining 
$\Lambda(z)$. Since Lemma~\ref{lem:ritz-values} also provides 
a way to scale the eigenvectors as required 
in Corollary~\ref{cor:f-res}, this allows us to evaluate 
\eqref{eq:cauchy-int} by only accessing the 
block upper Hessenberg form of the	projected 
matrix $A_j^{U,Z}$. 
\end{remark}
The previous result can be used to provide computable a posteriori 
error bounds for the approximation error of the action of 
a matrix function on a block vector. Let us consider a diagonalizable matrix $A$ with spectral decomposition $A=\sum_{h=1}^n\lambda_hx_hy_h^*$; then we may rewrite \eqref{eq:cauchy-int} as
\begin{align*}
f(A)B-F_j&=\sum_{i=1}^{js} \left[
f(A)   - \varphi(\theta_i) f(\theta_i) \varphi(A)^{-1}  \right] 
(A - \theta_i I)^{-1} 
(\Lambda^U(A) \circ B) v_i w_i^*\\
&=	
\upd{
\sum_{h=1}^nx_hy_h^*\sum_{i=1}^{js}\frac{f(\lambda_h)\varphi(\lambda_h)-\varphi(\theta_i)f(\theta_i)}{\lambda_h -\theta_i}(\varphi(A)^{-1}\Lambda^U(A) \circ B)v_iw_i^*}\\
&= \upd{\sum_{h=1}^nx_hy_h^* (\varphi(A)^{-1}\Lambda^U(A) \circ B)\sum_{i=1}^{js}\frac{f(\lambda_h)\varphi(\lambda_h)-\varphi(\theta_i)f(\theta_i)}{\lambda_h -\theta_i} v_iw_i^*.}
\end{align*}
By denoting with \upd{$L(\lambda_h):=\sum_{i=1}^{js}\frac{f(\lambda_h)\varphi(\lambda_h)-\varphi(\theta_i)f(\theta_i)}{\lambda_h -\theta_i}v_iw_i^*\in\mathbb C^{s\times s}$}, we have
\begin{align*}
\mathrm{vec}\left(f(A)B-F_j\right)&= \sum_{h=1}^n\left(L(\lambda_h)^T\otimes x_hy_h^*\right)\mathrm{vec}(\upd{\varphi(A)^{-1}\Lambda^U(A) \circ B}).
\end{align*}
Finally, we note that
\begin{align*}
\norm{f(A)B-F_j}_F&\le\norm{\upd{\varphi(A)^{-1}\Lambda^U(A) \circ B}}_F\norm{\sum_{h=1}^n\left(L(\lambda_h)^T\otimes x_hy_h^*\right)}_2\\ &=\norm{\upd{\Pi_{U,Z}U_{j+1}\Gamma_{j+1}^UR_B}}_F\norm{\sum_{h=1}^n(I\otimes X)\left(L(\lambda_h)^T\otimes e_he_h^*\right)(I\otimes X^{-1})}_2,
\end{align*}
where $X=[x_1,\dots,x_n]\in\mathbb C^{n\times n}$.
This leads to the following result.
\begin{corollary}\label{cor:aposteriori-matfun1}
  Under the assumptions of Theorem~\ref{thm:residual-rat}, 
  if $A_j^{U,Z}$ has simple eigenvalues
  and  $A$ is diagonalizable then
  \begin{align*}
    \| E_j \|_F &:= \| f(A) B - F_j \|_F 
    \leq \kappa_{\mathrm{eig}}(A)\norm{\upd{\Pi_{U,Z}U_{j+1}\Gamma_{j+1}^U R_B}}_F\max_{h=1,\dots,n}\norm{L(\lambda_h)}_2
  \end{align*}
  where $\kappa_{\mathrm{eig}}$ indicates the two norm condition number of an eigenvector matrix of $A$, and $L(\lambda_h)=\sum_{i=1}^{js}\frac{\upd{\varphi(\lambda_h)}f(\lambda_h)-\varphi(\theta_i)f(\theta_i)}{\lambda_h -\theta_i}v_iw_i^*$.
\end{corollary}

Computing the bounds from Corollary~\ref{cor:aposteriori-matfun1} 
requires some a-priori knowledge of the spectrum of $A$. Since 
we expect that the eigenvalues of $A$ are not explicitly known, 
we propose to sample the matrix-valued function $L(\lambda)$ over a set containing the spectrum of $A$.  \upd{If evaluated on a grid of $n_G$ points, the overall cost of computing the 
maximum of the $\| L(\lambda_h) \|_2$ 
is $\mathcal O(n_Gjs^3+j^3s^3)$. Finally, evaluating $\norm{\Pi_{U,Z}U_{j+1}\Gamma_{j+1}^U R_B}_F$ costs $\mathcal O(njs^2)$ in the general case and $\mathcal O(s^3)$ in the Galerkin case as we have $\norm{\Pi_{U,Z}U_{j+1}\Gamma_{j+1}^U R_B}_F= \norm{\Gamma_{j+1}^U R_B}_F$.}

\begin{remark}
The quantities $L(\lambda_h)$ may be written in terms of function of the projected matrix $A_{j}^{U,Z}$. In view of Lemma~\ref{lem:ritz-values}, we have that 
$w_i, v_i$ are obtained from the last and first blocks of (left and right) eigenvectors	$w_{H,i}, v_{H,i}$ 
for a block upper Hessenberg form $H = Q^* A_j^{U,Z} Q$ by 
\[
  w_i = R_B^* E_1^* w_{H,i}, \quad v_i = R_B^{-1} E_j^* v_{H,i}. 
\]
This follows from Remark~\ref{rem:monic-eigenvectors}, noting that 
the block Hessenberg form $H$ constructs the eigenvectors for the block 
characteristic polynomial with respect to $E_1$, and instead we need those 
for the block characteristic polynomial with respect to $E_1 R_B$.

We also remark that the leading coefficient of $\Lambda^U(z)$ is irrelevant 
in Corollary~\ref{cor:aposteriori-matfun1}, and one can use any block 
characteristic polynomial with respect to $E_1 R_B$ to evaluate the bound. 
Substituting the expressions for $w_i, v_i$ in the formula of Corollary~\ref{cor:aposteriori-matfun1} yields 
$$
L(\lambda) = 
  \Lambda^U(\lambda)\cdot R_B^{-1}E_j^*F_\lambda(H) E_1R_B = 
  \Lambda^U(\lambda)\cdot R_B^{-1}E_j^*Q^*F_\lambda(A_{j}^{U,Z})QE_1R_B,
$$
where $F_\lambda(z)= \frac{f(\lambda)\varphi(\lambda)-f(z)\varphi(z)}{(\lambda-z)}$. 
\end{remark}	 

We will see in Section~\ref{sec:numerical-matfun} that the bound from Corollary~\ref{cor:aposteriori-matfun1} can present some 
stability issues in the rational case. 
For this reason, 
we investigate another error formula, and the associated a posteriori bound, 
obtained by replacing \eqref{eq:res-H-rat} in \eqref{eq:f(a)res}. We will 
see in Section~\ref{sec:numerical-matfun} that this error bound is 
informative for all tested examples. 

\begin{corollary}\label{cor:residual-matfun2}
Under the assumptions of Theorem~\ref{thm:residual-rat}
and $A$ diagonalizable with eigendecomposition $A=\sum_{j=1}^n\lambda_jx_jy_j^*$, we have 
\[
  f(A)B - F_j = \sum_{h = 1}^n
     x_h y_h^* \Pi_{U,Z}U_{j+1} \Gamma_{j+1}^U 
     E_j^* K_{U,j}^{-1} 
     \left[ 
    f(A_j^{U,Z}) - f(\lambda_h) I
    \right](A_j^{U,Z} - \lambda_h I)^{-1} E_1 R_B,
\]
and the error norm verifies the inequality 
\[
  \|{f(A)B - F_j}\|_F \leq 
  \gamma \max_{h = 1, \ldots, n} 
  \| \Gamma_{j+1}^U E_j^*  K_{U,j}^{-1}
  \left[ 
   f(A_j^{U,Z}) - f(\lambda_h)I
   \right](A_j^{U,Z} - \lambda_h I)^{-1} E_1 R_B\|_2,
\]
where \upd{$\gamma:=  \kappa_{\mathrm{eig}}(A)\norm{\Pi_{U,Z}U_{j+1}\Gamma_{j+1}^U}_F$}, and $\kappa_{\mathrm{eig}}(A)$ is the $2$-norm condition number of an eigenvector matrix of $A$. 
\end{corollary}

\begin{proof}
Replacing \eqref{eq:res-H-rat} in \eqref{eq:cauchy-int}, we get 
\begin{align*}
  f(A) B - F_j 
  &= \frac{1}{2\pi i} \int_{\Gamma} (zI - A)^{-1} f(z) \Pi_{U,Z}U_{j+1} \Gamma_{j+1}^U 
  E_j^* K_{U,j}^{-1} 
    (zI - A_j^{U,Z})^{-1} E_j^* R_B  dz. 
\end{align*}
By means of the spectral decomposition 
$A = \sum_{h} \lambda_h x_h y_h^*$, we have 
\[
  f(A) B - F_j = \sum_h x_h y_h^* \Pi_{U,Z}U_{j+1}  \Gamma_{j+1}^U 
  E_j^* K_{U,j}^{-1} 
   \left( \frac{1}{2\pi i}\int_{\Gamma} 
    \frac{f(z)}{z-\lambda_h} 
    (zI - A_j^{U,Z})^{-1} dz\right)  E_j^* R_B. 
\]
Then, the residue formula yields 
\[
  f(A) B - F_j = \sum_{h = 1}^n
  x_h y_h^* \Pi_{U,Z}U_{j+1}  \Gamma_{j+1}^U 
  E_j^* K_{U,j}^{-1}
    \left[ 
     f(A_j^{U,Z}) - f(\lambda_h) I
     \right](A_j^{U,Z} - \lambda_h I)^{-1} E_1 R_B.
\]
Let us denote by \upd{$M_h :=   
E_j^* K_{U,j}^{-1}
\left[ 
 f(A_j^{U,Z}) - f(\lambda_h) I
 \right](A_j^{U,Z} - \lambda_h I)^{-1} E_1R_B$}, 
 then
 \upd{\begin{align*}
  \left\| \sum_{h = 1}^n
  x_h y_h^* \Pi_{U,Z}U_{j+1} \Gamma_{j+1}^U M_h \right\|_F  
  &= 
  \left\| \sum_{h = 1}^n
  (M_h^T \otimes x_h y_h^*) \mathrm{vec} (\Pi_{U,Z}U_{j+1} \Gamma_{j+1}^U) \right\|_2 \\
  &\leq \norm{\Pi_{U,Z}U_{j+1} \Gamma_{j+1}^U}_F \left\| \sum_{h = 1}^n
  (M_h^T \otimes x_h y_h^*) \right\|_2.
 \end{align*}}
 Applying a similarity transformation with the matrix $I \otimes [ x_1 \dots x_n ]$, 
 and a perfect shuffle\footnote{\upd{With perfect shuffle we mean the similarity transformation $\Pi( M_h^T \otimes x_h y_h^*)\Pi^* =  x_h y_h^*\otimes M_h^T$, for an appropriate permutation matrix $\Pi$.}} to the matrix  $\sum_{h = 1}^n
 (M_h^T \otimes x_h y_h^*)$ we get a block diagonal matrix with diagonal blocks equal to 
 $M_h$. Hence, we get the claim. 
\end{proof}
\upd{When $s=1$, $A$ is symmetric positive definite, and a Galerkin projection is employed, Corollary~\ref{cor:residual-matfun2} can be easily deduced from the proof of Theorem~4.8 in \cite{benzi2023computation}, that applies to the Krylov approximation of bilinear forms $b^Tf(A)b$.}
\begin{remark}
The upper bound in Corollary~\ref{cor:residual-matfun2} allows us to efficiently compute a posteriori error estimates. For instance, in the Galerkin case with \upd{a Hermitian} matrix $A$ whose spectrum is contained in $[a,b]\subset \mathbb R$, the inequality simplifies to
\begin{equation}\label{eq:aposteriori-matfun}
\| f(A) B - F_j \|_F\le \upd{\|\Gamma_{j+1}^U\|_F} \max_{\lambda\in[a,b]}\|  E_j^*  K_{U,j}^{-1}F( A_j^U, \lambda)
E_1 \upd{R_B} \|_2,
\end{equation}
where $F(z,\lambda):= 
\frac{f(z) - f(\lambda)}{z-\lambda}$. In practice, the maximum over $[a,b]$ in the right-hand side of \eqref{eq:aposteriori-matfun} is replaced with the maximum over a grid of $n_G$ points in $[a,b]$. \upd{The overall cost of evaluating the upper bound is $\mathcal O(n_Gj^2s^3 +j^3s^3)$. In the Petrov-Galerkin case an additional cost of $\mathcal O(njs^2)$ comes from evaluating $\norm{\Pi_{U,Z}U_{j+1}\Gamma_{j+1}^U}_F$.}
\end{remark}

\subsection{Numerical experiments}
\label{sec:numerical-matfun}

We now validate the error bounds of Corollary~\ref{cor:residual-matfun2}
and Corollary~\ref{cor:aposteriori-matfun1} on a set of test problems. 
Our bounds involve the condition number of the eigenvector matrix of $A$, 
and are not representative of the actual error when the matrix is highly 
non-normal. This is often the case for bounds for the norm of a 
matrix function that involve only the spectrum of the argument. An alternative, 
which we have not yet explored, are bounds based on spectral sets such as the 
numerical range or the $\epsilon$-pseudospetrum \cite{crouzeix2017numerical,trefethen2020spectra}. In view of 
the above reasons, in our numerical experiments 
we only consider Hermitian and normal matrices. 

\upd{The numerical tests involve randomly generated matrices; the bounds and 
the errors reported have been averaged over $10$ runs for each test. 
The code is public and  
available at \url{https://github.com/numpi/block-krylov-matfun}. We rely on \texttt{rktoolbox} \cite{berljafa2014rational} 
to construct the rational Krylov subspaces.}

\subsubsection{Galerkin approximation of the matrix exponential} 
Let us consider the discretization of the 1D Laplace operator 
with zero Dirichlet boundary conditions, and diffusion coefficient 
$K = 10^{-3}$. This yields  
$A = K (n+1)^2 \mathrm{tridiag}(1, -2, 1) \in \mathbb{R}^{n \times n}$, and we take as $B \in \mathbb R^{n \times 5}$ 
 a random matrix with i.i.d Gaussian entries, normalized to have $\norm{B}_F = 1$. 
We set $n = 1000$, 
and we employ a polynomial Krylov subspace to obtain a Galerkin 
approximation of $e^{\Delta t A} B$, with $\Delta t = 0.01$. 

We run Arnoldi for $j = 1, \ldots, 20$ steps, and for each value of 
$j$ we report the approximation error and the upper bound given 
by Corollary~\ref{cor:aposteriori-matfun1} and Corollary~\ref{cor:residual-matfun2}.

\begin{figure}[h]
\centering
\begin{tikzpicture}
  \begin{semilogyaxis}[
    xlabel={$j$},
    ylabel={},
    legend pos=north east,
    legend style={at={(1.05,1)}, anchor=north west},
    grid=major,
    width=.7\textwidth,
    height=0.45\textwidth
  ]
  \addplot[no marks, blue, very thick] table [x index=0, y index=1, col sep=space] {test1_paper.dat};
  \addlegendentry{Corollary~\ref{cor:residual-matfun2}}
  \addplot[no marks, red, very thick] table [x index=0, y index=2, col sep=space] {test1_paper.dat};
  \addlegendentry{Corollary~\ref{cor:aposteriori-matfun1}}
  \addplot[mark=square*, very thick, teal, mark size=1.5pt] table [x index=0, y index=3, col sep=space] {test1_paper.dat};
  \addlegendentry{$\|\exp(A)B - F_j\|_F$}
  \end{semilogyaxis}
\end{tikzpicture}
\caption{Galerkin approximation error for the action of the matrix exponential
onto a random block vector, and a posteriori bounds. The results are averaged 
over $10$ runs.}
\label{fig:test1_paper}
\end{figure}
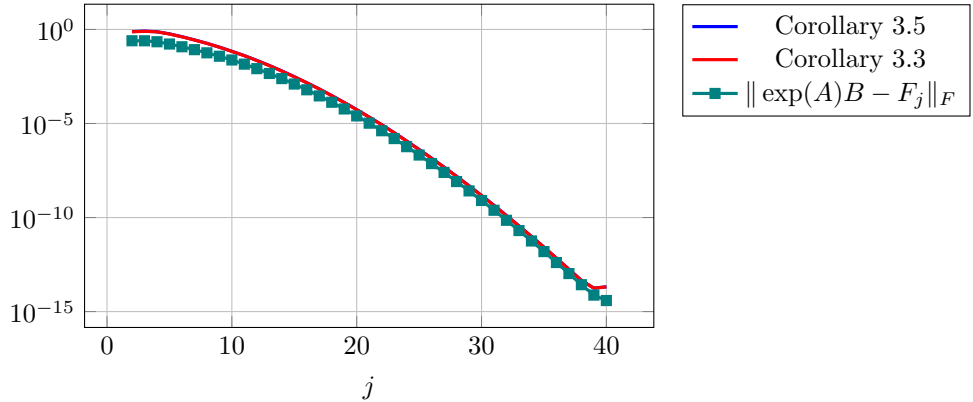

The error bounds are estimated by sampling the matrix-valued functions
on a grid of $100$ points over the spectral interval 
$[\lambda_{\min}(A), \lambda_{\max}(A)]$. The results 
are reported in Figure~\ref{fig:test1_paper}. The bounds 
look indistinguishable, and both provide estimates of the true error 
within an order of magnitude.

\begin{figure}[h]
\centering
\begin{tikzpicture}
  \begin{semilogyaxis}[
    legend columns=-1,
    xlabel={$j$},
    ylabel={},
    legend pos=north west,
    legend style={at={(0,1.05)}, anchor=south west},
    grid=major,
    width=.55\textwidth,
    height=0.45\textwidth,
    ymax = 1,
    ymin = 1e-20,
    name=plotleft
  ]
  \addplot[no marks, blue, very thick] table [x index=0, y index=1, col sep=space] {test2_paper.dat};
  \addlegendentry{Cor.~\ref{cor:residual-matfun2}}
  \addplot[no marks, red, very thick] table [x index=0, y index=2, col sep=space] {test2_paper.dat};
  \addlegendentry{Cor.~\ref{cor:aposteriori-matfun1}}
  \addplot[no marks, black, dashed, very thick] table [x index=0, y index=4, col sep=space] {test2_paper.dat};
  \addlegendentry{Cor.~\ref{cor:aposteriori-matfun1} (quad)}	
  \addplot[mark=square*, very thick, teal, mark size=1.5pt] table [x index=0, y index=3, col sep=space] {test2_paper.dat};
  \addlegendentry{$\|A^{-\frac 12}B - F_j\|_F$}

  \end{semilogyaxis}
  \begin{semilogyaxis}[
    xlabel={$j$},
    ylabel={},
    legend pos=north west,
    legend style={at={(0,1.05)}, anchor=south west},			
    yticklabels={},
    grid=major,
    width=.55\textwidth,
    height=0.45\textwidth,
    ymax = 1,
    ymin = 1e-20,
    anchor=west,
    at={($(plotleft.east)+(0.5cm,0)$)}
  ]
  \addplot[no marks, blue, very thick] table [x index=0, y index=1, col sep=space] {test3_paper.dat};
  \addplot[no marks, red, very thick] table [x index=0, y index=2, col sep=space] {test3_paper.dat};
  \addplot[no marks, black, dashed, very thick] table [x index=0, y index=4, col sep=space] {test3_paper.dat};
  \addplot[mark=square*, very thick, teal, mark size=1.5pt] table [x index=0, y index=3, col sep=space] {test3_paper.dat};
  \end{semilogyaxis}
\end{tikzpicture}
\caption{Galerkin approximation error for the action of the inverse square root
onto a random block vector, and a posteriori bounds. The results are averaged 
over $10$ runs. The plot on the left refers to the matrix $A_1$, while the plot on the right refers to the matrix $A_2$.}
\label{fig:test2_paper}
\end{figure}
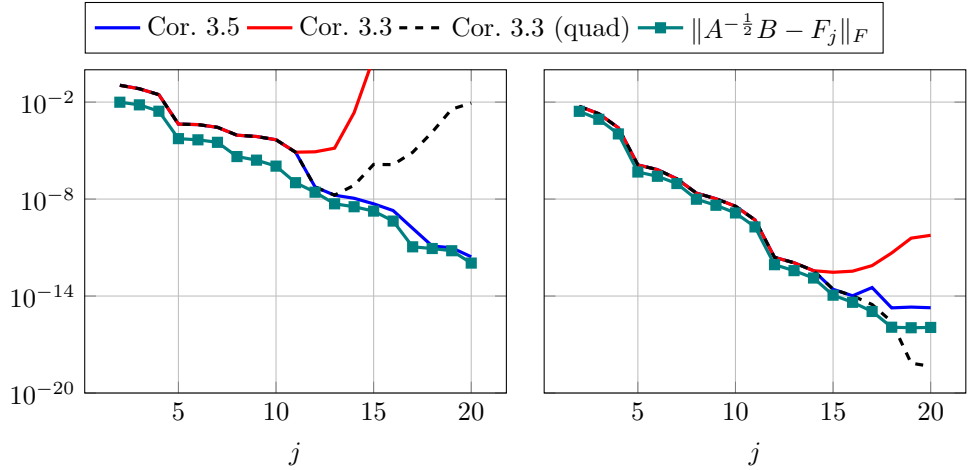
\subsubsection{Galerkin approximation of the matrix inverse square root}
\label{sec:2d-laplace}
\label{sec:numexp-galerkin-rational-insqrt}
We now conduct an experiment to test the a posteriori bounds in the rational Krylov setting. We consider $f(z)=z^{-\frac 12}$, and two matrix arguments, $A_1$, and $A_2$, with different condition numbers. We denote by $T=(n+1)^2\mathrm{tridiag(-1, 2,-1)}\in\mathbb C^{n\times n}$,
and we set 
$
A_1=I_n\otimes T+ T\otimes I_n$
and $A_2= A_1+(n+1)^2 I_{n^2}$
for $n=50$.
As in the previous section, we generate $B\in\mathbb R^{n\times 5}$ at random with unit Frobenius norm, and we use a rational Krylov subspace $\mathcal{RK}^\square_j(A,B,\Sigma)$ to get a Galerkin approximation of $A_i^{-\frac 12}B$, $i=1,2$. The set of poles $\Sigma$ is generated as described in \cite[Section 3.5]{massei2021rational} with the method of equidistributed sequences. The convergence history of the approach, and the a posteriori bounds arising from Corollary~\ref{cor:aposteriori-matfun1}, and Corollary~\ref{cor:residual-matfun2} are reported in Figure~\ref{fig:test2_paper}. 
Corollary~\ref{cor:residual-matfun2} provides a quite tight estimate of the 
approximation error for both matrices. However,  the evaluation of the upper bound in 
Corollary~\ref{cor:aposteriori-matfun1} suffers from numerical instability. To verify this, we added to the plot the bound computed 
with quadruple precision (by means of the Advanpix Multiprecision Computing Toolbox (\url{https://www.advanpix.com/}). This delays or eliminates the divergence.


\subsubsection{Petrov-Galerkin approximation of the matrix exponential}

Let us consider the Petrov-Galerkin approximation of the exponential
of a normal matrix $A \in \mathbb C^{n \times n}$, with $n = 1024$ and 
complex eigenvalues distributed as in 
Figure~\ref{fig:test_petrov} (left). More in detail, the eigenvalues are 
chosen as $\lambda_{i,j} = \rho_i e^{i \theta_j}$, where 
$\rho_i$ ranges in the set of logarithmically spaced points 
between $10^{-3}$ and $1$,
and $\theta_j$ are equispaced between $-\frac \pi 2$ and $\frac \pi 2$, 
for $i = 1, \ldots, 32$, and $j = 1, \ldots, 32$. 
\upd{We choose as $\Z_j$ an orthogonal basis of $\mathcal K_j^\square(A^T, C)$, where the} block vectors $B,C \in \mathbb C^{n \times 5}$ are generated randomly as in the previous experiments. Since $A$ is normal and the block vectors are chosen with a distribution
invariant under unitary transformations, without  loss of generality 
we take $A$ to be diagonal. The a posteriori bounds are sampled over a sectorial grid of 
$50 \times 50$ points with logarithmically \upd{distributed} moduli, and 
uniformly distributed arguments.
\begin{figure}[h]
\centering
\begin{tikzpicture}
  \begin{axis}[
    xlabel={Real part},
    ylabel={Imaginary part},
    grid=major,
    width=.5\textwidth,
    height=0.45\textwidth,
    legend pos=north east,
    legend style={at={(0,1.05)}, anchor=south west}
  ]
  \addplot[only marks, mark=*, teal, mark size=0.8pt] table [x index=0, y index=1, col sep=space] {test4_paper_petrov_eigs.dat};
  \addlegendentry{Eigenvalues}
  \end{axis}
  \begin{semilogyaxis}[
    xlabel={$j$},
    ylabel={},
    legend pos=north east,
    legend columns=-1,
    grid=major,
    width=.5\textwidth,
    height=0.45\textwidth,
    legend style={at={(1,1.05)}, anchor=south east},
    xshift = 6.5cm
  ]
  \addplot[no marks, blue, very thick] table [x index=0, y index=1, col sep=space] {test4_paper_petrov.dat};
  \addlegendentry{Cor.~\ref{cor:residual-matfun2}}
  \addplot[no marks, red, very thick] table [x index=0, y index=2, col sep=space] {test4_paper_petrov.dat};
  \addlegendentry{Cor.~\ref{cor:aposteriori-matfun1}}
  \addplot[mark=square*, very thick, teal, mark size=1.5pt] table [x index=0, y index=3, col sep=space] {test4_paper_petrov.dat};
  \addlegendentry{$\|\exp(A)B - F_j\|_F$}
  \end{semilogyaxis}
\end{tikzpicture}
\label{fig:test_petrov}

\caption{Petrov-Galerkin approximation error for the action of the matrix exponential
onto a random block vector, and a posteriori bounds. The matrix 
$A$ is diagonal with eigenvalues as in the left part of the Figure. 
The results are averaged 
over $10$ runs.}
\end{figure}
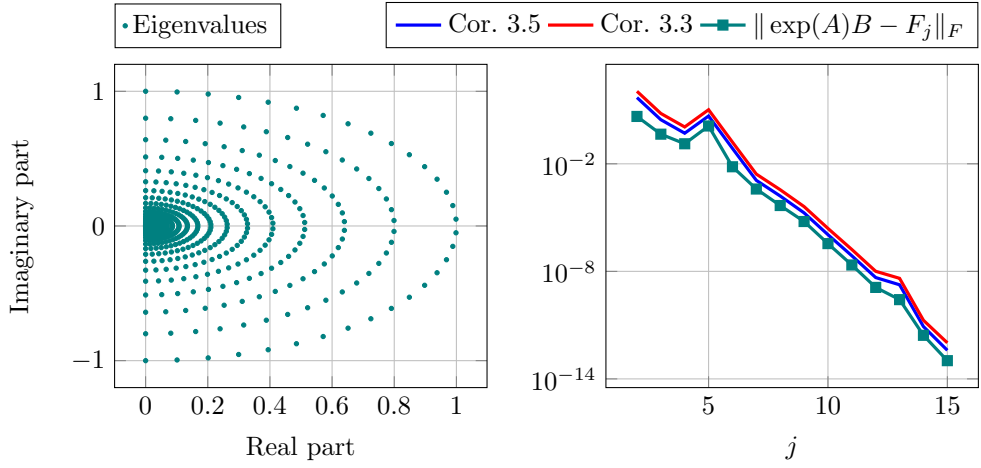
The results in Figure~\ref{fig:test_petrov} (right) show that both 
bounds are tight.   

\subsubsection{Petrov-Galerkin approximation of the matrix inverse square root}
We now test the computation of the matrix inverse square root with the Petrov-Galerkin method. We consider a similar diagonal matrix $A$ as in the previous experiment, and we approximate the action of $A^{-\frac 12}$ onto a random block vector $B$ with $s=5$ columns. The only difference is the imaginary part is now included in a 
smaller sectorial region of the complex, as displayed in Figure~\ref{fig:test_petrov2}. 
\upd{We choose the rational Krylov subspaces 
$\mathcal {RK}_j^\square(A,B,\Sigma)$ 
and $\mathcal {RK}_j^\square(A^*,C,\Sigma)$ for the Petrov-Galerkin approximation}, with poles 
generated as in \cite[Section 3.5]{massei2021rational}, by using 
as spectral interval $[2.5 \cdot 10^{-4}, 4]$. The latter interval is slightly 
larger than the minimum and \upd{maximum} real parts of the eigenvalues of $A$.
As previously, we sample the matrix-valued functions 
over a $50 \times 50$ sectorial grid.

\begin{figure}[h]
\centering
\begin{tikzpicture}
  \begin{axis}[
    xlabel={Real part},
    ylabel={Imaginary part},
    grid=major,
    width=.5\textwidth,
    height=0.45\textwidth,
    legend pos=north east,
    legend style={at={(0,1.05)}, anchor=south west}
  ]
  \addplot[only marks, mark=*, teal, mark size=0.8pt] table [x index=0, y index=1, col sep=space] {test5_paper_petrov_eigs.dat};
  \addlegendentry{Eigenvalues}
  \end{axis}
  \begin{semilogyaxis}[
    xlabel={$j$},
    ylabel={},
    legend pos=north east,
    legend columns=2
    grid=major,
    width=.5\textwidth,
    height=0.45\textwidth,
    legend style={at={(1,1.05)}, anchor=south east},
    ymax = 1e6,
    xshift = 6.5cm,
  ]
  \addplot[no marks, blue, very thick] table [x index=0, y index=1, col sep=space] {test5_paper_petrov.dat};
  \addlegendentry{Cor.~\ref{cor:residual-matfun2}}
  \addplot[no marks, red, very thick] table [x index=0, y index=2, col sep=space] {test5_paper_petrov.dat};
  \addlegendentry{Cor.~\ref{cor:aposteriori-matfun1}}
  \addplot[no marks, dashed, black, very thick] table [x index=0, y index=4, col sep=space] {test5_paper_petrov.dat};
  \addlegendentry{Cor.~\ref{cor:aposteriori-matfun1} (quad)}
  \addplot[mark=square*, very thick, teal, mark size=1.5pt] table [x index=0, y index=3, col sep=space] {test5_paper_petrov.dat};
  \addlegendentry{$\|A^{-\frac 12} B - F_j\|_F$}
  \end{semilogyaxis}
\end{tikzpicture}
\label{fig:test_petrov2}

\caption{Petrov-Galerkin approximation error for the inverse square root 
of $A$ applied to a random block vector, and a posteriori bounds. The matrix 
$A$ is diagonal with eigenvalues as in the left part of the Figure. 
The results are averaged 
over $10$ runs.}
\end{figure}
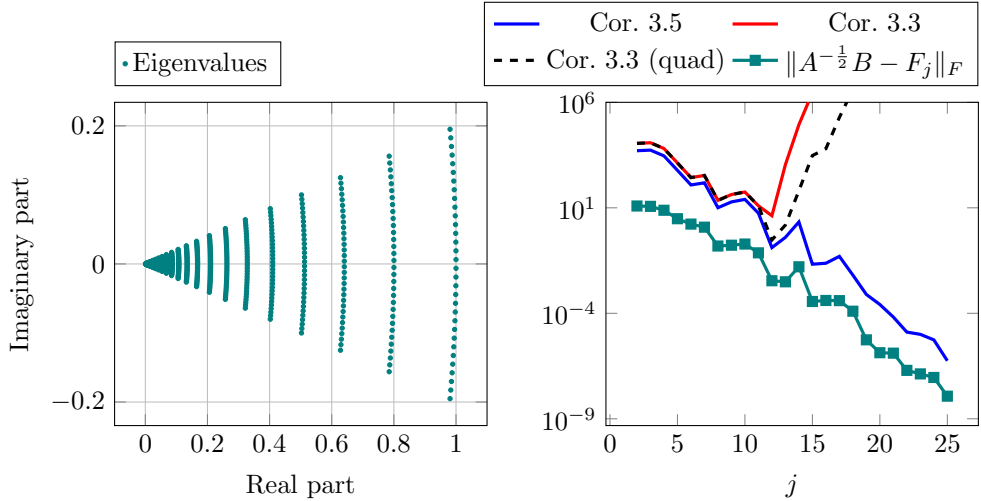

As in Section~\ref{sec:numexp-galerkin-rational-insqrt}, the bound from 
Corollary~\ref{cor:aposteriori-matfun1} suffers from numerical instability, 
and we report also its evaluation using quadruple precision. In both cases, 
the instability eventually appears, but we manage to get a reliable estimate 
of the error up to $10$ poles with doubles, 
and up to $12$ poles using quadruple precision. 
\upd{
\subsubsection{Computational times for the error bounds}
\label{sec:timings}
We conclude by testing the computational effort required to 
evaluate the bounds as stopping criterion in a Krylov method. We 
consider the 2D Laplace matrix of Section~\ref{sec:2d-laplace}, 
and the evaluation $e^{-A}B$ and $A^{-\frac 12} B$, 
with a polynomial and a rational Krylov subspace, respectively. 
In the rational case, the poles are chosen as in Section~\ref{sec:2d-laplace}.  
The number of iteration is fixed to $20$, we let the 
size of the matrix $A$ to grow from $400$ up to 
$10000$, and $B$ is always a random block vector with $5$ columns. For each test, we report the total time of the Krylov method
and the percentage spent evaluating the bounds. The results 
are reported in Table~\ref{tab:times} and showcase that the cost 
of the bound becomes negligible as the problem size increases. }

\begin{table}
	  \caption{Execution time needed for $20$ iterations of the block rational Krylov method for approximating $\exp(-A)B$ (left), and $A^{-\frac 12}B$ (right), with $A\in\mathbb R^{n^2\times n^2}$ taken as in Section~\ref{sec:2d-laplace}. The columns labeled as $t_B$, and $t_C$ indicate the total time employed by the method using as stopping criterion the bound from Corollary~\ref{cor:aposteriori-matfun1}, and Corollary~\ref{cor:residual-matfun2}, respectively. The columns labeled as $t_{\mathrm{bnd},B}$, and $t_{\mathrm{bnd},C}$ report the percentage of the total time spent for evaluating the stopping criteria.}  \label{tab:times} 
  \centering
  \pgfplotstableread[
    col sep=tab,
    header=false,
    columns={%
        N,
        tA, errA,
        tB, errB,
        tC, errC,
        tD, errD
    }
  ]{test_times.dat}\datatable
  \pgfplotstabletypeset[
    col sep=tab,
    columns/0/.style={column name={$n^2$},column type={c|}},
    columns/1/.style={fixed,precision=2,column name={$t_{B}$}},
    columns/2/.style={
        column name={$t_{\mathrm{bnd},B}$},
        fixed, precision = 1,
        postproc cell content/.append style={
            /pgfplots/table/@cell content/.add={
            }{\,\%}
        }
    },
    columns/3/.style={fixed,precision=2,column name={$t_C$}},
    columns/4/.style={
        column name={$t_{\mathrm{bnd},C}$},
        fixed, precision = 1,
        postproc cell content/.append style={
            /pgfplots/table/@cell content/.add={
            }{\,\%}
        },
        column type={c|}
    },
    columns/5/.style={fixed,precision=2,column name={$t_B$}},
    columns/6/.style={
        column name={$t_{\mathrm{bnd},B}$},
        fixed, precision = 1,
        postproc cell content/.append style={
            /pgfplots/table/@cell content/.add={
            }{\,\%}
        }
    },
    columns/7/.style={fixed,precision=2,column name={$t_C$}},
    columns/8/.style={
        column name={$t_{\mathrm{bnd},C}$},
        fixed, precision = 1,
        postproc cell content/.append style={
            /pgfplots/table/@cell content/.add={
            }{\,\%}
        }
    },
    every head row/.style={
        before row={
            \multicolumn{1}{c|}{} &
            \multicolumn{4}{c|}{$e^{-A} B$} &
            \multicolumn{4}{c}{$A^{- \frac 12} B$} \\
        },
        after row=\midrule
    }
  ]{\datatable}  

\end{table}

\section{Recurrence relations for block characteristic polynomials}
\label{sec:computational}
\upd{In this section we derive properties and recurrence relations for block characteristic polynomials coming from either a block upper Hessenberg matrix $H$ or matrices of the form $HK^{-1}$ with both $H,K$ block upper Hessenberg. The former will be useful in the analysis of block polynomial Krylov methods while the latter applies to the rational Krylov case. We always consider a block vector of the form $E_1 M$.} 

\subsection{The block upper Hessenberg case}

We now show that, given a sequence of embedded block upper Hessenberg matrices 
$\{ H_j \}_{j \geq 0}$, 
 defined as 
\begin{equation} \label{eq:Hj-clenshaw}
  H_j = \begin{bmatrix}
    \Phi_{1} & \Xi_{1,2}&\dots&\Xi_{1,j} \\ 
    \Gamma_2 & \Phi_2 & \ddots &\vdots \\ 
    & \ddots & \ddots & \Xi_{j-1,j} \\ 
    & & \Gamma_{j} & \Phi_j \\
  \end{bmatrix} = \left[ \begin{array}{ccc|c}
    \phantom{\Phi_{1}} & \phantom{\Xi_{1,2}} &\phantom{\dots} &\Xi_{1,j} \\ 
    \phantom{\Gamma_2} & H_{j-1} & \phantom{\ddots} &\vdots \\ 
    & \phantom{\ddots} & \phantom{\ddots} & \Xi_{j-1,j} \\ \hline 
    & & \Gamma_{j} & \Phi_j \\
  \end{array} \right] \in \mathbb R^{js \times js}, 
\end{equation}
their block characteristic polynomials
satisfy  a recurrence relation.

\begin{lemma} \label{lem:clenshaw}
  Let $P^{[j]}(\lambda)$ be the matrix polynomials recursively defined as
  \begin{align*}
    P^{[0]}&=I,\qquad P^{[1]}(\lambda) = \lambda I - \Phi_1,\\
    P^{[j]}(\lambda)&=  P^{[j-1]}(\lambda)(\lambda \Gamma_{j}^{-1}-\Gamma_j^{-1}\Phi_j)-\sum_{i=1}^{j-1}P^{[i-1]}(\lambda)(\Gamma_i^{-1}\Xi_{i,j}) ,
  \end{align*}
  where we set $\Gamma_1=I$. Then, $M^{-1} P^{[j]}(\lambda)$ is a block characteristic polynomial for 
  $H_j$ with respect to $E_1M$, and its leading coefficient is given by 
  $M^{-1} \Gamma_{2}^{-1}\dots\Gamma_{j}^{-1}$.
\end{lemma}
\begin{proof}
  We note that the claim is equivalent to showing that $P^{[j]}(\lambda)$ is a block 
  characteristic polynomial for $H_j$ with respect to $E_1$, and leading 
  coefficient $\Gamma_{2}^{-1}\dots\Gamma_{j}^{-1}$. The result can be verified by a 
  direct computation for $j=0,1$.
  
  Before performing the induction step, 
  we prove again by induction over $i$ that 
  $\left( P^{[i-1]}(H_j) \circ E_1 \right) \Gamma_i^{-1} = E_i$
  holds for all $i \leq j$. For $i = 1,2$, this can be verified directly 
  using the definition of $P^{[0]}(\lambda)$ and $P^{[1]}(\lambda)$. 
  Otherwise, 
  \begin{align*}
    P^{[i-1]}(H_j) \circ E_1 &= H_j \left[ P^{[i-2]}(H_j) \circ E_1\right] \Gamma_{i-1}^{-1}
    - \left[ P^{[i-2]}(H_j) \circ E_1\right] \Gamma_{i-1}^{-1}\Phi_{i-1} \\ 
    &- \sum_{h = 1}^{i-2} \left[ P^{[h-1]}(H_j) \circ E_1\right] \Gamma_{h-1}^{-1}\Xi_{h,i-1} \\ 
    &= H_j E_{i-1} - E_{i-1} \Phi_{i-1} - \sum_{h = 1}^{i-2} E_h \Xi_{h,i-1} = E_{i} \Gamma_i, 
  \end{align*}
  where the last equality follows from the block structure of $H_j$ in the $(i-1)$th column. 
  We conclude by proving that $P^{[j]}(H_j)\circ E_1 = 0$. The recurrence
  relation yields 
  \begin{align*}
    P^{[j]}(H_j) \circ E_1 &= 
    H_j \left(
    P^{[j-1]}(H_j) \circ E_1
    \right) \Gamma_j^{-1} - 
    (P^{[j-1]}(H_j) \circ E_1) \Gamma_j^{-1} \Phi_j \\ 
    &-\sum_{i = 1}^{j-1} (P^{[i-1]}(H_j) \circ E_1) \Gamma_i^{-1} \Xi_{ij} \\ 
    &= H_j E_j - E_j \Phi_j - \sum_{i = 1}^{j-1}E_i \Xi_{ij} = 0,
  \end{align*}
  where again the last equality comes from the structure of the last column 
  of $H_j$. 
\end{proof}

Lemma~\ref{lem:clenshaw} allows us to give a concise proof 
of Lemma~\ref{lem:ritz-values}. 
\begin{proof}[Proof of Lemma~\ref{lem:ritz-values}]
  Let us consider the matrix pencil
  \[
    \lambda I - H = \begin{bmatrix}
    \lambda I - \Phi_{1} & -\Xi_{1,2}&\dots&-\Xi_{1,j} \\ 
    -\Gamma_2 & \lambda I -\Phi_2 & \ddots &\vdots \\ 
    & \ddots & \ddots & -\Xi_{j-1,j} \\ 
    & & -\Gamma_{j} & \lambda I - \Phi_j \\
  \end{bmatrix}.
  \]
  Note that the block entry in position $(1,1)$ is 
  $P^{[1]}(\lambda)$. 
  Since $\Gamma_2$ is invertible, 
  we can use block Gaussian elimination to annihilate 
  the entries in the second block row using the first 
  block column. A direct computation shows that 
  the first block row is 
  \[
    \begin{bmatrix}
    P^{[1]}(\lambda) & P^{[2]}(\lambda) & 
      P^{[1]}(\lambda) \Gamma_2^{-1} \Xi_{2,3} -\Xi_{1,3} & 
      \dots & 
      P^{[1]}(\lambda) \Gamma_2^{-1} \Xi_{2,j} -\Xi_{1,j} 
    \end{bmatrix}
  \]
  Proceeding by induction,  after $j - 1$ steps 
  the reduced matrix has the form 
  \[
    \begin{bmatrix}
    P^{[1]}(\lambda) & \dots & P^{[j-1]}(\lambda)& P^{[j]}(\lambda) \\
    -\Gamma_2 \\
    & \ddots \\
    && -\Gamma_j 
    \end{bmatrix}.
  \]
  This implies that $H$ is a linearization for $P(\lambda)$, 
  since all block Gauss moves are unimodular transformations
  \cite{glr}, and in particular $H$ and 
  $P^{[j]}(\lambda)$ have the same eigenvalues. 
  Indeed, there exist unimodular matrices $G(\lambda)$ 
  and $F(\lambda)$ such that 
  \[
    G(\lambda) (\lambda I - H) F(\lambda) = \begin{bmatrix}
    0 & \dots & 0 & P^{[j]}(\lambda) \\
    -\Gamma_2 \\
    & \ddots \\
    && -\Gamma_j 
    \end{bmatrix}.
  \]
  Hence, the right and left eigenvectors of $H$ are of the 
  form 
  \[
    F(\lambda) \begin{bmatrix}
      0 & \dots & 0 & v_j^*
    \end{bmatrix}^*, \qquad 
    \begin{bmatrix}
      w_1 & 0 & \dots & 0
    \end{bmatrix} G(\lambda), 
  \]
  respectively. Then, we observe that 
  $E_1^* [ G(\lambda) (\lambda I - H) F(\lambda) ]' E_j = 
  (P^{[j]})'(\lambda)$, and 
  \[
    [ G(\lambda) (\lambda I - H) F(\lambda) ]' = 
    G'(\lambda) (\lambda I - H) F(\lambda) + 
    G(\lambda) F(\lambda) + 
    G(\lambda) (\lambda I - H) F'(\lambda).
  \]
  Left multiplying by $\begin{bmatrix}
    w_1^* & 0 & \dots & 0
  \end{bmatrix}$ and right multiplying by $\begin{bmatrix}
    0 & \dots & 0 & v_j^*
  \end{bmatrix}^*$ yields 
  \begin{align*}
    \begin{bmatrix}
      w_1^* & 0 & \dots & 0
    \end{bmatrix} [ G(\lambda) (\lambda I - H) F(\lambda) ]'
    \begin{bmatrix}
      0 \\ \vdots \\ 0 \\ v_j
    \end{bmatrix} &= 
    w_1^* (P^{[j]})'(\lambda) v_j.
  \end{align*}
  Since $w^* = \begin{bmatrix}
    w_1^* & 0 & \dots & 0
  \end{bmatrix} G(\lambda)$ and 
  $v = F(\lambda) \begin{bmatrix}
    0 & \dots & 0 & v_d^*
  \end{bmatrix}^*$, the claim holds. 
\end{proof}

\subsection{The block upper Hessenberg pencil case}
\upd{
In the case of rational Krylov subspaces, we can generalize the 
recursion in Lemma~\ref{lem:clenshaw} using the rational 
Arnoldi relation \eqref{eq:rational-arnoldi-relation}. More specifically, the matrix pencil
\begin{equation} \label{eq:rational-linearization}
  \lambda K_{U,j} - H_{U,j} = \begin{bmatrix}
    \Xi_{1,1}(\lambda) & \Xi_{1,2}(\lambda) & \ldots & \Xi_{1,j}(\lambda) \\
    \Gamma_2 (\lambda - \sigma_1) & \Xi_{2,2}(\lambda) & \dots & \Xi_{2,j}(\lambda) \\
    & \ddots & \ddots & \vdots \\ 
    && \Gamma_j (\lambda - \sigma_j) & \Xi_{j,j}(\lambda) \\
  \end{bmatrix}
\end{equation}
encodes all the information needed to evaluate  the rational matrix polynomial.

\begin{theorem}\label{thm:rational-clenshaw}
Assume that $\Gamma_i$ is invertible for $i=2,\dots,j$, and let $P^{[i]}(\lambda)$ be the matrix polynomials defined by recurrence 
as follows:
\begin{align*}
    P^{[1]}(\lambda) &= \Xi_{1,1}(\lambda),\\
    P^{[j]}(\lambda)&=  
    \Xi_{1,j}(\lambda) \prod_{s = 1}^{j-1} (\lambda - \sigma_s) -
    \sum_{i=1}^{j-1} P^{[i]}(\lambda) \Gamma_{i+1}^{-1} \Xi_{i+1,j}(\lambda)
    \prod_{s = i+1}^{j-1} (\lambda - \sigma_s).
\end{align*}
Then, the pencil $\lambda K_{U,j} - H_{U,j}$ is such that:
\begin{enumerate}
  \item Every eigenvalue $\theta$ of the matrix polynomial 
  $P^{[j]}(\lambda)$ is an eigenvalue of the pencil with the same 
  multiplicity.
  \item If 
    $v = [\times, \ldots, \times, v_j^*]^*$ and 
    $w = [w_1^*, \times, \ldots, \times]^*$ 
    are right and left eigenvectors of the 
    pencil
    corresponding to the eigenvalue $\theta$, then 
    $v_j$ and $w_1$ are right and left eigenvectors of 
    $P^{[j]}(\lambda)$ 
    corresponding to $\theta$, and
    \[
      w_i^* (P^{[j]})'(\theta) v_i = 
      w^* K_{U,j} v.
    \]
  \item $P^{[j]}(\lambda)$ is a block characteristic polynomial
    of $H_{U,j} K_{U,j}^{-1}$ with respect to $E_1$. 
    \item The leading coefficient of $P^{[j]}(\lambda)$ is $(E_j^TK_{U,j}^{-1}E_1)^{-1}$.
\end{enumerate}
\end{theorem}

\begin{proof}
  We prove the result by showing the existence of 
  a unimodular rational matrix 
  $E(\lambda)$ and a matrix polynomial $F(\lambda)$ such that
  \[
    E(\lambda) (\lambda K_{U,j} - H_{U,j}) F(\lambda) = 
    \begin{bmatrix}
      0 & \dots & 0 & P^{[j]}(\lambda) \\ 
      \Gamma_2 (\lambda - \sigma_1) &&& 0\\ 
      & \ddots && \vdots \\ 
      && \Gamma_j (\lambda - \sigma_j) & 0 \\ 
    \end{bmatrix}
  \]
  Similarly to the proof of Lemma~\ref{lem:ritz-values}, 
  we use block Gaussian elimination with $\Gamma_2 (\lambda - \sigma_1)$,
  to eliminate all the entries on the same block row. However, 
  to stay in the ring of matrix polynomials, we multiply 
  the block columns from $2$ to $j$ by $\lambda - \sigma_1$; 
  repeating this process for all the block rows
  we can prove by induction that
  \[
    (\lambda K_{U,j} - H_{U,j}) F(\lambda) = \begin{bmatrix}
      P^{[1]}(\lambda) & \dots & P^{[j-1]}(\lambda) & P^{[j]}(\lambda) \\ 
      \Gamma_2 (\lambda - \sigma_1) &&& 0\\ 
      & \ddots && \vdots \\ 
      && \Gamma_j (\lambda - \sigma_j) & 0 \\ 
    \end{bmatrix}
  \]
  and 
  \[
    F(\lambda) = \begin{bmatrix}
      I & \times & \dots & \times  \\ 
      & (\lambda - \sigma_1) I &  \ddots & \vdots  \\ 
      & & \ddots & \times \\ 
      & & & \prod_{i = 1}^{j-1} (\lambda - \sigma_i) I \\
    \end{bmatrix}
  \]
  Similarly, we can construct $E(\lambda)$ using block row operations 
  to eliminate the entries in the first block row; we force $E(\lambda)$ to 
  have identities on the diagonal blocks, and rational matrices 
  on the upper triangular part in the first block row. More 
  specifically, 
  \[
    E(\lambda) = \begin{bmatrix}
      I & -P^{[1]}(\lambda) \Gamma_2^{-1} (\lambda - \sigma_1)^{-1} & 
        \dots & -P^{[j-1]}(\lambda) \Gamma_j^{-1} (\lambda - \sigma_j)^{-1} \\ 
        & I \\ 
        && \ddots \\ 
        &&& I
    \end{bmatrix}.
  \]
  Let $\theta$ be an eigenvalue of $\lambda K_{U,j} - H_{U,j}$; then, 
  we know $\theta \neq \sigma_i$ for all $i = 1, \ldots, j-1$, and 
  therefore $E(\lambda)$ and $F(\lambda)$ define a local unimodular 
  transformation at $\lambda = \theta$ \cite[Theorem 2.2]{corless}. Therefore, 
  the eigenvalue of the pencil and of $P^{[j]}(\lambda)$ coincide and ---if $v_j$ and $w_1$ are right and left eigenvectors of $P^{[j]}(\lambda)$--- 
  $[w_1^*,\ 0,\ \dots,\ 0] E(\theta)$ and 
  $F(\theta) [0,\ \dots,\ 0,\ v_j^*]^*$ are left and right eigenvectors
  of the pencil, respectively.

  As in the proof of Lemma~\ref{lem:ritz-values}, deriving 
  the relation between eigenvectors of the pencil and the 
  matrix polynomials yields 
  $w_1^* (P^{[j]})'(\theta) v_j = w_1^* K_{U,j} v_j$. 

  To prove that $P^{[j]}(\lambda)$ is a block
  characteristic polynomial we prove that whenever $K$
  is invertible, 
  \[
    E(\lambda) (\lambda K - H) F(\lambda) = \begin{bmatrix}
      & P^{[j]}(\lambda) \\ 
      G(\lambda) 
    \end{bmatrix}
  \]
  and for all eigenvalues $\theta$ of $\lambda K - H$
  we have $\det E(\theta), \det F(\theta), \det G(\theta)\ne 0$, 
  and $P^{[j]}(H K^{-1}) \circ E_1 = 0$. To obtain this, 
  we note that given a left eigenpair $(\theta, w)$ 
  of $\lambda K - H$ such that $w^* H K^{-1} = \theta w^*$,
  we have 
  \[
    w^* P^{[j]}(HK^{-1}) \circ E_1 =
    w_1^* P^{[j]}(\theta), \qquad 
    w_1 := w^* E_1.
  \]
  Therefore, 
  \begin{align*}
    0 &= w^* (\theta I  - HK^{-1}) = 
     w^* E(\theta)^{-1} E(\theta) (\theta I - HK^{-1}) K F(\theta)F(\theta)^{-1} \\
     &=
     w^* E(\theta)^{-1} \begin{bmatrix}
      0 & P^{[j]}(\theta) \\
      G(\theta)
    \end{bmatrix} F(\theta)^{-1}.
  \end{align*}
  Hence, we have $w^* E(\theta)^{-1} 
  \left[ \begin{smallmatrix}
    0 \\ G(\theta)
  \end{smallmatrix} \right] = 0$, that implies 
  $w^*E(\theta)^{-1} = [\tilde w_1^*, 0, \ldots, 0] $, and in turn 
  $\tilde w_1^* P^{[j]}(\theta) = 0$. In addition, note that 
  $w_1 = w^* E_1 = [\tilde w_1^*, 0, \ldots, 0] E(\theta) E_1 = \tilde w_1^*$. Since $HK^{-1}$ is diagonalizable, this holds for a basis 
  of left eigenvectors, and we conclude that 
  $P^{[j]}(HK^{-1}) \circ E_1 = 0$.
  Finally, we write $P^{[j]}(\lambda)$ as
  $$
  P^{[j]}(\lambda)=E_1^T(\lambda K_{U,j}-H_{U,j}) F(\lambda)E_j,
  $$ and thanks to the antitriangular  structure $(P^{[j]}(\lambda))^{-1}=E_j^TF(\lambda)^{-1}(\lambda K_{U,j}-H_{U,j})^{-1}E_1$. Since $F(\lambda)$ is upper triangular, and we know its diagonal entries, we have
  $$
  (P^{[j]}(\lambda))^{-1}=\prod_{i=1}^{j-1}(\lambda -\sigma_i)^{-1}E_j^T (\lambda K_{U,j}-H_{U,j})^{-1}E_1.
  $$	
  Therefore, the inverse of the leading coefficient
  of $P^{[j]}(\lambda)$ is obtained taking the limit
  \begin{align*}
  	\lim_{\lambda\to\infty} \lambda^j (P^{[j]}(\lambda))^{-1} &=\lim_{\lambda\to\infty} \prod_{i=1}^{j-1}(1 -\frac{\sigma_i}{\lambda})^{-1} ( K_{U,j}-\lambda^{-1}H_{U,j})^{-1}E_1 = E_j^TK_{U,j}^{-1}E_1.
  \end{align*}
\end{proof}
\upd{We remark that the recurrence relation of Theorem~\ref{thm:rational-clenshaw} corresponds to the one in \cite[Theorem 4.3, Algorithm 4.1]{elsworth20}, up to dividing the $P^{[i]}(\lambda)$ by $\Pi_{h=1}^i(\lambda-\sigma_i)$, and considering a rational Arnoldi decomposition scaled so that $\Gamma_i=I_s$ for $i=2,\dots, j$.}

We are now ready to prove Corollary~\ref{cor:final-formula}.

\begin{proof}[Proof of Corollary~\ref{cor:final-formula}]
Assume first that the set of poles $\Sigma$ does not contain poles at infinity.
Since $P^{[j]}$ is 
the block characteristic polynomial with respect to $E_1$, the 
one with respect to $E_1 R_B$,  
is given by $\Lambda^U=R_{B}^{-1} P^{[j]} R_B$, 
and therefore has leading coefficient 
$R_{B}^{-1} (E_j^TK_{U,j}^{-1}E_1)^{-1} R_B$. 
In the Galerkin case, since $A_j^U = H_{U,j} K_{U,j}^{-1}$, combining 
this observation with  Lemma~\ref{lem:rat-lanczos} gives a proof of Corollary~\ref{cor:final-formula}.
In the Petrov-Galerkin scenario, the projected matrix is given by 
$A_j^{U,Z} = H_{U,j} K_{U,j}^{-1} + (Z_j^* U_j)^{-1} Z_j^* U_{j+1} \Gamma_j^U E_j^* K_{U,j}^{-1}$; in particular, $A_j^{U,Z} = \tilde H K_{U,j}^{-1}$ 
for the block upper Hessenberg matrix
\[
  \tilde H = H_{U,j} + (Z_j^* U_j)^{-1} Z_j^* U_{j+1} \Gamma_j^U E_j^*.
\]
Applying again Lemma~\ref{lem:rat-lanczos} with Theorem~\ref{thm:rational-clenshaw} we obtain a proof of 
Corollary~\ref{cor:final-formula}.

If the set $\Sigma$ contains one or more poles at infinity, we cannot directly apply Theorem~\ref{thm:rational-clenshaw}, and we use a continuity argument. Note that, given $\zeta\in\mathbb C$ a complex number that is neither a Ritz value nor a pole, the quantity $\varphi(A)^{-1}\Lambda^U(A)\circ B$ corresponds, up to a right multiplication with an invertible matrix, to the Petrov-Galerkin residual of $(\zeta I-A)X=B$, see Theorem~\ref{thm:residual-rat}. Since we want to study the behavior of the rational Krylov subspace when one or more poles go to infinity, it is natural to consider the poles $\sigma_i=\frac{\alpha_i}{\beta_i}$ where the $[\alpha_i:\beta_i]\in\mathbb P^{1}(\mathbb C)$. Then, in the Arnoldi process one can solve linear systems with matrix coefficients $(\alpha_i I -\beta_i A)$ in place of $(\sigma_iI -A)$. Therefore, the Petrov-Galerkin residual and the quantities $\Pi_{U,Z}, U_{j+1},$ and $\Gamma_{j+1}^U$, depend continuously on the poles, over the Riemann sphere, and this gives the claim. Note that, the case with all poles at infinity was already covered by the combination of Lemma~\ref{lem:poly-lanczos} and Lemma~\ref{lem:clenshaw}.
\end{proof}
}

\section{Conclusions and outlook}
\label{sec:conclusions}
\upd{In this work we have introduced theoretical and algorithmic tools to analyze the convergence of block Krylov methods; we relied on these} to propose two novel \upd{formulas}, contained in Corollary~\ref{cor:aposteriori-matfun1} and Corollary~\ref{cor:residual-matfun2}, for the error of Petrov-Galerkin approximations of matrix functions with block Krylov subspaces. Based on the latter formula, we could provide two a posteriori error bounds for quantities of the form $f(A)B$, that can be evaluated with a computational cost that scales linearly with respect to the size of $A$. The practical evaluation of the bounds requires  
the knowledge of a complex region enclosing the spectrum of $A$. The bound of Corollary~\ref{cor:aposteriori-matfun1} can be affected by numerical instability, 
and this can be mitigated using higher precision in its computation. This might be 
still convenient when dealing with large scale matrices, since the bound only requires 
computations projected version of the original matrix.

Several aspects deserve further investigations and remain for future study. For instance, the analysis of the numerical stability of the proposed procedures is an open issue. Another direction, is to explore the use of block characteristic polynomials for preconditioning block Krylov methods, in analogy to what has been done for the single vector case \cite{loe2022toward}. Finally, it would be of interest to leverage \eqref{eq:remainder-transfer} or \eqref{eq:remainder-transfer2} to propose and test error indicators for approximating transfer functions of MIMO.

\appendix 
\section{A Clenshaw algorithm to evaluate $P(A)\circ B$}
The results in the previous sections pose a natural computational challenge: evaluate 
quantities of the form 
\begin{equation} \label{eq:P(A)b}
	Y = P(A) \circ B, 
\end{equation}
where $P$ is a block characteristic polynomial of 
a matrix $N\in \mathbb C^{js\times js}$  with respect to a block vector $W\in\mathbb C^{js\times s}$. In the case of polynomial Krylov subspaces, 
the matrix $N$ is block upper Hessenberg, and $W = E_1 M$ for some
matrix $M\in\mathbb C^{s\times s}$. This structure is convenient when designing a numerical 
procedure to evaluate \cref{eq:P(A)b}, and in view of Lemma~\cref{lem:clenshaw-general} we can always reduce to this case. 
In practice, we do not rely on the block Arnoldi procedure, as in the proof of Lemma~\ref{lem:clenshaw-general}; rather, 
by means of block Householder reflectors \cite{rotella1999block}  
we compute a unitary matrix $Q$ such that $Q^* N Q = H$ and 
$Q^* W = E_1 M$, and the block characteric polynomial for $H$ with respect 
to $E_1 M$ is equal to the block characteristic polynomial of $N$ with respect to $W$. For this reason, in this section we assume to be working with 
a block upper Hessenberg $H$ and a block vector $E_1 M$. 

We propose an algorithm to evaluate 
\eqref{eq:P(A)b} based on 
two observations: $(i)$ the block characteristic polynomials of block upper Hessenberg matrices with 
respect to vectors $E_1M$ satisfy a  recurrence relation, and $(ii)$ this allows us to 
construct a block Clenshaw rule for the evaluation of $P(A)\circ B$. This is a natural extension of what 
is already well known for the non-block case \cite{golub-meurant}. 

The previous result  suggests a numerical procedure (Clenshaw algorithm) to evaluate $P(A) \circ B$. If the monic version of the block characteristic polynomial 
is desired, it is sufficient to right multiply the result by $\Gamma_j \ldots \Gamma_2 M$ 
at the end. The procedure is reported in Algorithm~\ref{alg:clenshaw}, 
and is described for general matrices $N$ and block vectors $W$, 
employing Lemma~\ref{lem:clenshaw-general}. 
\begin{algorithm}
	\caption{Evaluate $P(A) \circ B$, for $A\in\mathbb C^{n\times n}$, and $B\in\mathbb C^{n\times s}$ where $P$ is a block 
		characteristic polynomial of 
		$N\in\mathbb C^{js\times js}$ with respect to $W\in\mathbb C^{js\times s}$.}\label{alg:clenshaw}
	\begin{algorithmic}
		\Procedure{BlockClenshaw}{$N, W, A, B$, monic}
		\State Compute $Q,H$,$M$ such that $Q^* N Q = H, Q^* W = E_1 M$ 
		\Comment{Lemma~\ref{lem:clenshaw-general}}
		\State $j \gets $ number of block rows and cols in $H$
		\State $P^{[0]} \gets BM^{-1}$
		\For{$i = 1, \ldots j$}
		\State $P^{[i]} \gets AP^{[i-1]} \Gamma_i^{-1} - 
		P^{[i-1]} \Gamma_i^{-1} \Phi_i - 
		\sum_{h = 1}^{i-1} P^{[h-1]} \Gamma_h^{-1} \Xi_{h,i}$
		\EndFor
		\If{monic}
		\For{$i = j, j-1, \ldots 2$}
		\State $P^{[j]} \gets P^{[j]} \Gamma_i$
		\EndFor
		\State $P^{[j]} \gets P^{[j]} M$
		\EndIf
		\State \Return $P^{[j]}$
		\EndProcedure
	\end{algorithmic}
\end{algorithm}

Let us analyze the computational cost of Algorithm~\ref{alg:clenshaw}. The  reduction to block upper Hessenberg form requires $\mathcal O(j^3s^3)$ arithmetic operations. Moreover, at  iteration $i$ of the for loop, the procedure evaluates:
\begin{itemize}
	\item one  product involving the $A\in\mathbb C^{n\times n}$ and a block vector with $s$ columns,
	\item $\mathcal O(i)$ products between $n\times s$ and $s\times s$ matrices,
	\item $\mathcal O(i)$ right division between $n\times s$ and $s\times s$ matrices.
\end{itemize}
Therefore, the cost of generating $P^{[1]},\dots,P^{[j]}$ is $\mathcal O(\mathcal C_A js +nj^2s^2)$, where $\mathcal C_A$ denotes the complexity of a matrix-vector product with the matrix $A$. Finally, the $\mathcal O(i)$ right multiplications needed to get the evaluation of the monic block characteristic cause an additional $\mathcal O(njs^2)$. Thus, the overall complexity of Algorithm~\ref{alg:clenshaw} is $\mathcal O(\mathcal C_Ajs+nj^2s^2+j^3s^3)$. Concerning the memory consumption, the method stores all the block vectors $P^{[i]}$, and this has a cost of $\mathcal O(njs)$. 

We remark that Algorithm~\ref{alg:clenshaw} can also be  used to evaluate the matrix polynomial $P(\lambda)\in\mathbb R[\lambda]^{s\times s}$ for a scalar value $\lambda$, by means of the identity $P(\lambda)=P(\lambda I_s)\circ I_s$. 
In this case the cost of Algorithm~\ref{alg:clenshaw} reduces to $\mathcal O(j^3s^3)$ ($\mathcal O(j^2s^3)$ if the reduction to block upper Hessenberg form is not needed) for the CPU time and $\mathcal O(js)$ for the memory usage.

\bibliographystyle{siamplain}
\bibliography{references} 
\end{document}